\documentclass[11pt]{amsart}
\input xy \xyoption{all}
\usepackage{amsfonts}
\usepackage{amssymb}
\usepackage{amscd}
\usepackage{xypic}

\DeclareMathOperator\C{\mathbb C}

\DeclareMathOperator\Hom{\mathrm Hom}
\DeclareMathOperator\Gr{\mathrm Gr}
\DeclareMathOperator\Res{\mathrm Res}
\DeclareMathOperator\I{\mathcal I}
\DeclareMathOperator\F{\mathcal F}

\DeclareMathOperator\G{\mathfrak G}
\DeclareMathOperator\PP{\mathbb P}
\DeclareMathOperator\codim{codim}
\DeclareMathOperator\rk{rk}

\newtheorem{fact}{Fact}[section]
\newtheorem{lemma}[fact]{Lemma}
\newtheorem{theorem}[fact]{Theorem}
\newtheorem{definition}[fact]{Definition}
\newtheorem{example}[fact]{Example}
\newtheorem{rremark}[fact]{Remark}
\newenvironment{remark}{\begin{rremark} \rm}{\end{rremark}}
\newtheorem{proposition}[fact]{Proposition}
\newtheorem{corollary}[fact]{Corollary}
\newtheorem{conjecture}[fact]{Conjecture}
\newcommand{\dblq}{{/\!/}}

\def\L{\mathcal{L}}
\def\G{\mathcal{G}}
\def\cE{\mathcal{E}}
\def\I{\mathcal{I}}

\def\cM{\mathcal{M}}

\def\cV{\mathcal{V}}

\def\cU{\mathcal{U}}
\def\cX{\mathcal{X}}
\def\cA{\mathcal{A}}
\def\cB{\mathcal{B}}
\def\cF{\mathcal{F}}
\def\cC{\mathcal{C}}
\def\H{\mathcal{H}}
\def\Pic0{{\rm Pic}^0(X)}

\def\mm{\overline{\mathcal{M}}}
\def\hh{\overline{\mathcal{H}}}

\def\PP{{\textbf P}}
\def\OO{\mathcal{O}}

\author[G. Farkas]{Gavril Farkas}
\address{Humboldt-Universit\"at zu Berlin, Institut f\"ur Mathematik, Unter den Linden 6
\hfill \newline\texttt{}
 \indent 10099 Berlin, Germany} \email{{\tt farkas@math.hu-berlin.de}}

\author[R. Rim\'anyi]{Rich\'ard Rim\'anyi}
\address{University of North Carolina at Chapel Hill, Department of Mathematics,  USA}
\email{\tt rimanyi@email.unc.edu}

\title{Quadric rank loci on moduli of curves and $K3$ surfaces}

\begin{document}

\begin{abstract}
Assuming that $\phi:\mbox{Sym}^2(\cE)\rightarrow \cF$ is a morphism of vector bundles on a variety $X$, we compute the class of the locus in $X$ where
$\mbox{Ker}(\phi)$ contains a quadric of prescribed rank. Our formulas have many applications to moduli theory: (i) we find a
 simple proof of Borcherds' result that the Hodge class on the moduli space of polarized $K3$ surfaces of fixed genus is of Noether-Lefschetz type,
(ii) we construct an explicit canonical divisor on the Hurwitz space parametrizing degree $k$ covers of $\PP^1$ from curves of genus $2k-1$, (iii) we provide
a closed formula for the Petri divisor on $\mm_g$ of canonical curves which lie on a rank $3$ quadric and (iv) construct
myriads of effective divisors of small slope on $\mm_g$.
\end{abstract}

\maketitle

\section{Introduction}

Let $X$ be an algebraic variety and let $\cE$ and $\cF$ be two vector bundles on $X$ having ranks $e$ and $f$ respectively. Assume we are given a morphism  of vector bundles $$\phi:\mbox{Sym}^2(\cE)\rightarrow \cF.$$
For a positive integer $r\leq e$, we define the subvariety of $X$ consisting of points for which $\mbox{Ker}(\phi)$ contains a quadric of corank at least $r$, that is, $$\overline{\Sigma}^r_{e,f}(\phi):=\Bigl\{x\in X: \exists \ 0\neq q\in \mbox{Ker}(\phi(x)) \ \mbox{ with } \  \mbox{rk}(q)\leq e-r\Bigr\}.$$
Since the codimension of the variety of symmetric $e\times e$-matrices of corank $r$ is equal to ${r+1\choose 2}$, it follows that the expected codimension of the locus $\overline{\Sigma}^r_{e,f}(\phi)$ is equal to ${r+1\choose 2}-{e+1\choose 2}+f+1.$ A main goal of this paper is to explicitly determine the cohomology class of this locus in terms of the Chern classes of $\cE$ and $\cF$. This is achieved for every $e,f$ and $r$  in Theorem \ref{thm:residue}, using a localized Atiyah-Bott type formula. Of particular importance in moduli theory is the case when this locus is expected to be a divisor, in which case our general formula has a very simple form:

\begin{theorem}\label{classdiv1}
We fix integers $0\leq r\leq e$ and set $f:={e+1\choose 2}-{r+1\choose 2}$. Suppose  $\phi:\mathrm{Sym}^2(\cE)\rightarrow \cF$ is  a morphism of vector bundles   over $X$. The
class of the virtual divisor $\overline{\Sigma}^r_{e,f}(\phi)$ is  given by the formula
$$[\overline{\Sigma}^r_{e,f}(\phi)]=A_e^r\Bigl(c_1(\cF)-\frac{2f}{e} c_1(\cE)\Bigr)\in H^2(X,\mathbb Q),$$
where
$$A_e^r:=\frac{{e\choose r}{e+1\choose r-1}\cdots {e+r-1\choose 1}}{{1\choose 0}{3\choose 1}{5\choose 2}\cdots {2r-1\choose r-1}}.$$
\end{theorem}
The quantity $A_e^r$ is the degree of the variety of symmetric $e\times e$-matrices of corank at least $r$ inside the projective space of all symmetric $e\times e$ matrices, see \cite{ht}.

\vskip 3pt

Before introducing a second type of degeneracy loci, we give a definition. If $V$ is a vector space, a pencil of quadrics $\ell\subseteq \PP(\mathrm{Sym}^2(V))$ is said to be \emph{degenerate} if the intersection of $\ell$ with the discriminant divisor $D(V)\subseteq \PP(\mathrm{Sym}^2(V))$ is non-reduced. We consider a morphism $\phi:\mbox{Sym}^2(\cE)\rightarrow \cF$ such that all kernels are expected to be pencils of quadrics and impose the condition that the pencil be degenerate.

\begin{theorem}\label{degpencint}
We fix integers $e$ and $f={e+1\choose 2}-2$ and let $\phi:\mathrm{Sym}^2(\cE)\rightarrow \cF$ be a morphism of vector bundles. The class of the virtual divisor $\mathfrak{Dp}:=\Bigl\{x\in X: \mathrm{Ker}(\phi(x)) \ \mbox{is a degenerate pencil}\Bigr\}$ equals
$$[\mathfrak{Dp}]=(e-1)\Bigl(ec_1(\cF)-(e^2+e-4)c_1(\cE)\Bigr)\in H^2(X,\mathbb Q).$$
\end{theorem}

Theorems \ref{classdiv1} and \ref{degpencint} are motivated by fundamental questions in moduli theory and in what follows we shall discuss some of these applications, which are treated at length
in the paper.

\vskip 3pt

\noindent {\bf{Tautological classes on moduli of polarized $K3$ surfaces.}}
Let $\cF_g$ be the moduli space of quasi-polarized $K3$ surfaces $[X,L]$ of genus $g$, that is, satisfying $L^2=2g-2$. We denote by $\pi:\cX\rightarrow \cF_g$ the universal $K3$ surface and choose a polarization line bundle $\L$ on $\cX$. We consider the Hodge class
$$\lambda:=c_1\bigl(\pi_*(\omega_{\pi})\bigr)\in CH^1(\cF_g).$$
Note that $CH^1(\cF_g)\cong H^2(\cF_g,\mathbb Q)$. Inspired by Mumford's definition of the $\kappa$ classes on $\cM_g$, for integers $a,b\geq 0$, Marian, Oprea and Pandharipande \cite{MOP} introduced the classes $\kappa_{a,b}\in CH^{a+2b-2}(\cF_g)$ whose definition we recall in Section \ref{sectk3}. In codimension $1$, there are two such classes, namely
$$\kappa_{3,0}:=\pi_*\Bigl(c_1(\L)^3\Bigr) \ \mbox{ and } \ \kappa_{1,1}:=\pi_*\Bigl(c_1(\L)\cdot c_2(\mathcal{T}_{\pi})\Bigr)\in CH^1(\cF_g).$$
Both these classes depend on the choice of $\L$, but the following linear combination
$$\gamma:=\kappa_{3,0}-\frac{g-1}{4} \kappa_{1,1}\in CH^1(\cF_g)$$
is intrinsic and independent of the polarization line bundle.

\vskip 3pt

For a general element $[X,L]\in \cF_g$ one has $\mbox{Pic}(X)=\mathbb Z \cdot L$. Imposing the condition that $\mbox{Pic}(X)$ be of rank at least $2$, one is led to the notion of
Noether-Lefschetz (NL) divisor on $\cF_g$. For non-negative integers $h$ and $d$, we denote by $D_{h,d}$ the locus of quasi-polarized $K3$ surfaces $[X,L]\in \cF_g$ such that there exists a primitive embedding of a rank $2$ lattice
$$\mathbb Z\cdot L\oplus \mathbb Z\cdot D\subseteq \mathrm{Pic}(X),$$ where $D\in \mbox{Pic}(X)$ is a class such that $D\cdot L=d$ and $D^2=2h-2$. From the Hodge Index Theorem
 $D_{h,d}$ is empty unless $d^2-4(g-1)(h-1)>0$. Whenever non-empty, $D_{h,d}$ is pure of codimension $1$.

\vskip 4pt

Maulik and Pandharipande \cite{MP} conjectured that $\mbox{Pic}(\cF_g)$ is spanned by the Noether-Lefschetz divisors $D_{h,d}$. This has been recently proved in \cite{BLMM} using deep automorphic techniques. Note that the rank of $\mbox{Pic}(\cF_g)$ can become arbitrarily large and understanding all the relations between NL divisors remains a daunting task. Borcherds \cite{Bo} using  automorphic forms on $O(2,n)$ has shown that the Hodge class $\lambda$ is supported on NL divisors. A second proof of this fact, via Gromov-Witten theory, is due to Pandharipande and Yin, see \cite{PY} Section 7. Using  Theorem \ref{classdiv1}, we find very simple and explicit Noether-Lefschetz representatives of both classes  $\lambda$ and $\gamma$. Our methods are  within the realm of algebraic geometry and we use no automorphic forms.

\vskip 4pt

We produce relations among tautological classes on $\cF_g$ using the projective geometry of embedded $K3$ surfaces of genus $g$. We study geometric conditions that  single out \emph{only} NL special K3 surfaces. Let us first consider the divisor in $\cF_g$ consisting of $K3$ surfaces which lie on a rank $4$ quadric. We fix a $K3$ surface $[X,L]\in \cF_g$ with $g\geq 4$ and let
$\varphi_L:X\rightarrow \PP^g$ be the morphism induced by the polarization $L$.  One computes $h^0(X,L^{\otimes 2})=4g-2$. Assuming that the image $X\subseteq \PP^g$ is projectively normal (which holds under very mild genericity assumptions, see again Section \ref{sectk3}), we observe that the space $I_{X,L}(2)$ of quadrics containing $X$ has the following dimension:
$$\mbox{dim } I_{X,L}(2)=\mbox{dim } \mbox{Sym}^2 H^0(X,L)-h^0(X,L^{\otimes 2})={g-2\choose 2}.$$
This equals the codimension of the space of symmetric $(g+1)\times (g+1)$ matrices of rank $4$. Therefore the condition that
$X\subseteq \PP^g$ lie on a rank $4$ quadric is expected to be divisorial on $\cF_g$. This expectation is easily confirmed in Proposition \ref{nldiv3}, and we
are led to the divisor:
$$D_g^{\mathrm{rk} 4}:=\Bigl\{[X,L]\in \cF_g: \exists \ 0\neq q\in I_{X,L}(2), \ \ \mbox{rk}(q)\leq 4\Bigr\}.$$
\begin{theorem}\label{rank4intro}
Set $g\geq 4$. The divisor $D_g^{\mathrm{rk} 4}$ is an effective combination of NL divisors and its class is
$$[D_g^{\mathrm{rk} 4}]=A_{g+1}^{g-3}\Bigl((2g-1)\lambda+\frac{2}{g+1}\gamma \Bigr)\in CH^1(\cF_g).$$
\end{theorem}

In order to get a second relation between $\lambda$ and $\gamma$, we distinguish depending on the parity of $g$.  For odd genus $g$, we obtain a second relation between $\lambda$ and $\gamma$ by considering the locus of $K3$ surfaces $[X,L]\in \cF_g$ for which the embedded surface $\varphi_L: X\rightarrow \PP^g$  has a non-trivial middle linear syzygy. In terms of Koszul cohomology groups, we set
$$\mathfrak{Kosz}_g:=\Bigl\{[X,L]\in \cF_g: K_{\frac{g-1}{2},1}(X,L)\neq 0\Bigr\}.$$
For instance $\mathfrak{Kosz}_3$ consists of quartic $K3$ surfaces  for which the map $\mbox{Sym}^2 H^0(X,L)\rightarrow H^0(X,L^{\otimes 2})$ is not an isomorphism. Voisin's solution \cite{V1} of the generic Green Conjecture on syzygies of canonical curves ensures that $\mathfrak{Kosz}_g$ is a proper locus of NL type. She proved  that  for a $K3$ surface $[X,L]\in \cF_g$ with $\mbox{Pic}(X)=\mathbb Z\cdot L$,  the vanishing
$$K_{\frac{g-1}{2},1}(X,L)=0$$
holds, or equivalently, $[X,L]\notin \mathfrak{Kosz}_g$. We realize $\mathfrak{Kosz}_g$ as the degeneracy locus of a morphism of two vector bundles of the same
rank over $\cF_g$, whose Chern classes can be expressed in terms of $\kappa_{1,1}, \kappa_{3,0}$ and $\lambda$. We then obtain the following formula (see Theorem \ref{koszclass})
\begin{equation}\label{koszclassintro}
[\mathfrak{Kosz}_g]=\frac{4}{g-1}{g-4\choose \frac{g-3}{2}}\Bigl(\frac{(g-1)(g+7)}{2}\lambda+\gamma\Bigr)+\alpha\cdot [D_{1,1}]\in CH^1(\cF_g),
\end{equation}
where recall that $D_{1,1}$ is the NL divisor of $K3$ surfaces $[X,L]$ for which the polarization $L$ is not globally generated.
Theorems \ref{rank4intro} and \ref{koszclass} then quickly imply (in the case of odd $g$):

\begin{theorem}\label{borcherds}
Both tautological classes $\lambda$ and $\gamma$ on $\cF_g$ are of Noether-Lefschetz type.
\end{theorem}

\vskip 4pt

Theorem \ref{borcherds} is proved for even genus $g\geq 8$ in Section \ref{seclm} using two further geometric relations between tautological classes (in the spirit of Theorem \ref{rank4intro}) involving the geometry of rank $2$ Lazarsfeld-Mukai bundle $E_L$ one associates canonically to each $NL$-general polarized $K3$ surface $[X,L]\in \cF_g$. The vector bundle
$E_L$ satisfies $\mbox{det}(E_L)=L$ and $h^0(X,E_L)=\frac{g}{2}+2$ and has already been put to great use in \cite{La}, \cite{Mu}, or \cite{V1}. A direct proof of Theorem \ref{borcherds} when $g\leq 10$ has already appeared in \cite{GLT}.

\vskip 4pt

In Section \ref{gitss} we discuss an application of Theorem \ref{rank4intro} to the Geometric Invariant Theory of $K3$ surfaces. The second Hilbert point $[X,L]_2$ of a suitably general polarized $K3$ surface $[X,L]$ is defined as the quotient
$[X,H]_2:=\Bigl[\mathrm{Sym}^2 H^0(X,L)\longrightarrow H^0(X,L^{\otimes 2})\longrightarrow 0\Bigr] \in \Gr\Bigl(\mathrm{Sym}^2 H^0(X,H), 4g-2\Bigr)$. We establish the following result:

\begin{theorem}\label{gitintro}
The second Hilbert point of a  polarized $K3$ surface $[X,L]\in \cF_g \setminus D_g^{\mathrm{rk} 4}$ is semistable.
\end{theorem}

Note that a similar result at the level of canonical curves has been obtained in \cite{FJ}.
\vskip 5pt

\noindent{\bf{The Petri class on $\mm_g$.}}

\vskip 3pt

A non-hyperelliptic canonical curve $C\subseteq \PP^{g-1}$ of genus $g$ is projectively normal and lies on precisely ${g-2\choose 2}$ quadrics. This number equals the codimension of the locus of symmetric $g\times g$-matrices of rank $3$. The condition that $C$ lie on a rank $3$ quadric in its canonical embedding is divisorial and leads to the Petri divisor $\mathcal{GP}_g$ of curves $[C]\in \cM_g$, having a pencil $A$ such that the Petri map $$\mu(A): H^0(C,A)\otimes H^0(C,\omega_C\otimes A^{\vee})\rightarrow H^0(C,\omega_C)$$ is not injective. Using Theorem \ref{classdiv1}, we establish the following result:

\begin{theorem}\label{petriclass}
The class of the compactified Petri divisor $\widetilde{\mathcal{GP}}_g$ on $\mm_g$ is given by the formula
$$[\widetilde{\mathcal{GP}}_g]=A_g^{g-3}\Bigl(\frac{7g+6}{g}\lambda-\delta\Bigr)\in CH^1(\mm_g).$$
\end{theorem}

Here $\lambda$ is the Hodge class on $\mm_g$ and $\delta$ denotes the total boundary divisor. The Petri divisor splits into components $D_{g,k}$, where $\lfloor \frac{g+2}{2}\rfloor \leq k\leq g-1$, depending on the degree of the (base point free) pencil $A$ for which the Petri map $\mu(A)$ is not injective. With a few notable exception when $k$ is extremal, the individual classes $[\overline{D}_{g,k}]\in CH^1(\mm_g)$ are not known. However, we predict a simple formula for the multiplicities of $\overline{D}_{g,k}$ in the expression of $[\widetilde{\mathcal{GP}}_g]$, see Conjecture \ref{multpetri}.

\vskip 6pt

\noindent{\bf{Effective divisors on Hurwitz spaces.}}

\vskip 3pt

We fix an integer $k\geq 4$ and denote by $\H_k$ the Hurwitz space parametrizing degree $k$ covers $[f:C\rightarrow \PP^1]$ from a smooth curve of genus $2k-1$. The space $\H_k$ admits a compactification $\hh_k$ by means of admissible covers, which is defined to be the normalization of the space constructed by Harris and Mumford in \cite{HM}. We refer to \cite{ACV} for details. We denote by $\sigma:\hh_k\rightarrow \mm_{2k-1}$ the morphism assigning to each admissible cover the stabilization of the source curve. The image $\sigma(\hh_k)$ is the divisor  $\mm_{2k-1,k}^1$ consisting of $k$-gonal curves in $\mm_{2k-1}$, which was studied in great detail by Harris and Mumford \cite{HM} in the course of their proof that $\mm_g$ is general for large genus. The birational geometry of $\hh_k$ is largely unknown, see however \cite{ST} for some recent results.

\vskip 3pt

Let us choose a general point $[f:C\rightarrow \PP^1]\in \H_k$ and denote by $A:=f^*(\OO_{\PP^1}(1))\in W^1_k(C)$ the pencil inducing the cover. We consider the residual linear system  $L:=\omega_C\otimes A^{\vee}\in W_{3k-4}^{k-1}(C)$ and denote by $\varphi_L:C\rightarrow \PP^{k-1}$ the induced map. Under these genericity assumptions $L$ is very ample, $H^1(C,L^{\otimes 2})=0$  and the image curve $\varphi_L(C)$ is projectively normal. In particular,
$$\mbox{dim} \ I_{C,L}(2)=\mbox{dim } \mbox{Sym}^2 H^0(C,L)-h^0(C,L^{\otimes 2})={k-3\choose 2},$$
which equals the codimension of the space of symmetric $k\times k$ matrices of rank $4$. Imposing the condition that $C\subseteq \PP^{k-1}$ be contained in a rank $4$ quadric, we obtain a (virtual) divisor
$$\mathfrak{H}_k^{\mathrm{rk} 4}:=\Bigl\{[C,A]\in \H_k: \exists \ 0\neq q\in I_{C,\omega_C\otimes A^{\vee}}(2), \ \mbox{rk}(q)\leq 4\Bigr\}.$$

The condition $[C,A]\in \mathfrak{H}_k^{\mathrm{rk} 4}$ amounts to representing the canonical bundle $\omega_C$ as a sum
\begin{equation}\label{3pen}
\omega_C=A\otimes A_1\otimes A_2
\end{equation}
of \emph{three} pencils, that is, $h^0(C,A_1)\geq 2$ and $h^0(C,A_2)\geq 2$. To show that $\mathfrak{H}_k^{\mathrm{rk} 4}$ is indeed a divisor, it suffices to exhibit a point $[C,A]\in \H_k$ such that
(\ref{3pen}) cannot hold. To that end, we take a general polarized $K3$ surface $[X,L]\in \cF_{2k-1}$ carrying an elliptic pencil $E$ with $E\cdot L=k$ (that is, a general element of the NL divisor $D_{1,k}\subseteq \cF_{2k-1}$). If $C\in |L|$ is a smooth curve in the polarization class and $A=\OO_C(E)\in W^1_k(C)$, we check that one has an isomorphism $I_{C,\omega_C\otimes A^{\vee}}(2)\cong I_{X,L(-E)}(2)$ between the spaces of quadrics containing $C$ and $X\subseteq \PP^{k-1}$ respectively. Showing that this latter space contains no rank $4$ quadric becomes a lattice-theoretic problem inside $\mbox{Pic}(X)$, which we solve.

\vskip 3pt

We summarize  our results concerning $\mathfrak{H}_k^{\mathrm{rk} 4}$. We denote by $\lambda:=\sigma^*(\lambda)$ the Hodge class on $\hh_k$ and by $D_0$ the boundary divisor on $\hh_k$ whose general point corresponds to a $1$-nodal singular curve $C$ of genus $2k-1$ and a locally free sheaf $A$ of degree $k$ with $h^0(C,A)\geq 2$ (see Section \ref{hurw1} for details).

\begin{theorem}\label{hurwdiv}
For each $k\geq 6$, the locus $\mathfrak{H}_k^{\mathrm{rk} 4}$ is an effective divisor on $\H_k$. Away from the union of the boundary divisors $\sigma^{-1}(\Delta_i)$ where $i=1, \ldots, k-1$,  one has the relation
$$K_{\hh_k}=\frac{k-12}{k-6}\Bigl(7\lambda-[D_0]\Bigr)+\frac{k}{(k-6)A_k^{k-4}}[\overline{\mathfrak{H}}_k^{\mathrm{rk} 4}].$$
\end{theorem}

Theorem \ref{hurwdiv} follows from applying Theorem \ref{classdiv1} in the context of Hurwitz spaces to compute the class $[\overline{\mathfrak{H}}_k^{\mathrm{rk} 4}]$ in terms of certain tautological classes on $\hh_k$, see Theorem \ref{partclass}, then comparing with the formula we find for $K_{\hh_k}$ in terms of those same classes. Proving that $\mathfrak{H}_k^{\mathrm{rk} 4}$ is indeed a genuine divisor on $\hh_k$ is achieved in Theorem \ref{transvhur}.

\vskip 3pt

We mention the following consequence to the birational geometry of $\hh_k$.

\begin{theorem}\label{koddim}
For $k> 12$, there exists an effective $\mathbb Q$-divisor class $E$ on $\hh_k$  supported on the divisor $\sum_{i=1}^{k-1} \sigma^*(\Delta_i)$ of curves of compact type, such that  the class $K_{\hh_k}+E$ is big.
\end{theorem}

This result should be compared to the classical result \cite{HM} asserting that $\mm_{2k-1}$ is of general type for $k\geq 13$, whereas the Kodaira dimension of $\mm_{23}$ is at least $2$, see \cite{F4}.  Assuming that the singularities of $\hh_k$ impose no adjunction conditions (something one certainly expects), Theorem \ref{koddim} should imply that for $k> 12$ the Hurwitz space $\hh_k$ is a variety of general type.

\vskip 5pt

\noindent {\bf{Effective divisors of small slope on $\mm_g$.}}

\vskip 3pt

Theorem \ref{classdiv1} has multiple applications to the birational geometry of the moduli space of curves. Recall that if $\lambda, \delta_0, \ldots, \delta_{\lfloor \frac{g}{2}\rfloor}$ denote the standard generators of $\mbox{Pic}(\mm_g)$, then the \emph{slope} of an effective divisor $D\subseteq \mm_g$ such that $\Delta_i\nsubseteq \mbox{supp}(D)$ for all $i=0, \ldots, \lfloor \frac{g}{2}\rfloor$, is defined as $s(D):=\frac{a}{\mathrm{min}_i b_i}\geq 0$, where $[D]=a\lambda-\sum_{i=0}^{\lfloor \frac{g}{2}\rfloor} b_i\delta_i \in \mbox{Pic}(\mm_g)$. The slope of the moduli space, defined as the quantity $$s(\mm_g):=\mbox{inf}\Bigl\{s(D): D \mbox{ is an effective divisor of }\  \mm_g\Bigr\}$$ is a fundamental invariant encoding for instance the Kodaira dimension of the moduli space. For a long time it was conjectured  \cite{HMo} that $s(\mm_g)\geq 6+\frac{12}{g+1}$,  with equality if and only if $g+1$ is composite and $D$ is a Brill-Noether divisor on $\mm_g$ consisting of  curves $[C]\in \cM_g$ having a linear series $L\in W^r_d(C)$ with Brill-Noether number $\rho(g,r,d)=-1$. This conjecture has been disproved in \cite{F1}, \cite{F2} and \cite{Kh}, where for an infinite series of genera $g$  effective divisors of slope less than $6+\frac{12}{g+1}$ were constructed.  At the moment there is no clear conjecture concerning even the asymptotic behavior of $s(\mm_g)$ as $g$ is large, see also \cite{P}. For instance, it is not  clear that $\mbox{liminf}_{g\rightarrow \infty} s(\mm_g)>0$.

\vskip 3pt

Imposing the condition that a  curve $C$ of genus $g$ lie on a  quadric of prescribed rank in one of the  embeddings $\varphi_L:C\hookrightarrow \PP^r$  given by a linear system $L\in W^r_d(C)$ with Brill-Noether number $\rho(g,r,d):=g-(r+1)(g-d+r)=0$, we obtain an infinite sequence of effective divisors on $\mm_g$ of very small slope (see
condition (\ref{numconda}) for the numerical condition  $g$  has to satisfy).  Theorems \ref{pelda1} and \ref{pelda2} exemplify two infinite subsequences of such divisors on $\mm_{(4\ell-1)(9\ell-1)}$ and $\mm_{4(3\ell+1)(2\ell+1)}$ respectively, where $\ell\geq 1$.  We mention the following concrete example on $\mm_{24}$.

\begin{theorem}\label{24intro}
The following locus defined as
$$D_{7,3}:=\Bigl\{[C]\in \cM_{24}: \exists \ L\in W^7_{28}(C), \ \exists \ 0\neq q\in I_{C,L}(2), \ \mathrm{rk}(q)\leq 6\Bigr\}$$
 is an effective divisor on $\cM_{24}$. The slope of its closure $\overline{D}_{7,3}$ in $\mm_{24}$ is given by
$s(\overline{D}_{7,3})=\frac{34423}{5320}<6+\frac{12}{25}.$
\end{theorem}

Theorem \ref{transv24} establishes that $D_{7,3}$ is a genuine divisor on $\cM_{24}$. We show using \emph{Macaulay} that there exists a smooth curve $C\subseteq \PP^7$ of genus $24$ and degree $28$ which does not lie on a quadric of rank at most $6$ in $\PP^7$. Using the irreducibility of the space of pairs $[C,L]$, where $C$ is a smooth curve of genus $24$ and $L\in W^7_{28}(C)$, we conclude that $D_{7,3}\neq \cM_{24}$, hence $D_{7,3}$ is indeed  a divisor on $\cM_{24}$.

\vskip 4pt
Theorem \ref{degpencint} has applications to the slope of $\mm_{12}$. A general curve $[C]\in \cM_{12}$ has a finite number of embeddings $C\subseteq \PP^5$ of degree $15$. They are all residual to pencils
of minimal degree. The curve  $C\subseteq \PP^5$ lies on a pencil of quadrics and we impose the condition that one of these pencils be degenerate.

\begin{theorem}\label{sec12}
The locus of smooth curves of genus $12$ having a degenerate pencil of quadrics
$$\mathfrak{Dp}_{12}:=\Bigl\{[C]\in \cM_{12}:\exists \ L\in W^5_{15}(C) \mbox{ with } \PP\bigl({I_{C,L}}(2)\bigr) \ \mbox{degenerate}\Bigr\}$$
is an effective divisor. The slope of its closure $\overline{\mathfrak{Dp}}_{12}$ inside $\mm_{12}$ equals $s(\overline{\mathfrak{Dp}}_{12})=\frac{373}{54}<6+\frac{12}{13}$.
\end{theorem}

\vskip 3pt

\noindent {\bf{Acknowledgements:}} We warmly thank Rahul Pandharipande, Alessandro Verra and Claire Voisin for very interesting discussions related to this circle of ideas. We are thankful to Daniele Agostini for his help with the \emph{Macaulay} calculations appearing in this paper. We are most grateful to the referee of this paper, who spotted several inaccuracies and mistakes and whose many insightful remarks, concerning all sections of the paper, significantly improved it.

\vskip 4pt

\section{Equivariant fundamental classes, degeneracy loci}

\subsection{Equivariant fundamental class}
We consider a connected algebraic group $G$ acting on a smooth variety $V$, and let $\Sigma$ be an invariant subvariety. Then $\Sigma$ represents a fundamental cohomology class---denoted by $[\Sigma]$ or $[\Sigma\subseteq V]$---in the $G$-equivariant cohomology of $V$, namely
\[
[\Sigma]\in H^{2 \text{codim}(\Sigma\subseteq V)}_G(V).
\]
Throughout the paper we use cohomology with complex coefficients. There are several equivalent ways to define this fundamental cohomology class, see for example \cite{K1}, \cite{EG}, \cite{FR1}, \cite[8.5]{MS}, \cite{fultonnotes} for different flavours and different cohomology theories.

A particularly important case is when $V$ is a vector space and $\Sigma$ is an invariant cone. Then $[\Sigma]$ is an element of $H_G^*(\text{vector space})=H_G^*(\text{point})=H^*(BG)$, that is, $[\Sigma]$ is a $G$-characteristic class. This characteristic class has the following well known ``degeneracy locus'' interpretation.
Let $E\to M$ be a bundle with fiber $V$ and structure group $G$. Since $\Sigma$ is invariant under the structure group, the notion of {\em belonging to $\Sigma$} makes sense in every fiber. Let $\Sigma(E)$ be the union of $\Sigma$'s of all the fibers. Let $s$ be a sufficiently generic section. Then the fundamental cohomology class $[s^{-1}(\Sigma(E))\subseteq M]$ of the ``degeneracy locus'' $s^{-1}(\Sigma(E))$ in the {\em ordinary} cohomology $H^*(M)$ is equal to $[\Sigma]$ (as a $G$-characteristic class) of the bundle $E\to M$.

\subsection{Examples}
We recall two well known formulas for some equivariant cohomology classes. The second one will be used in Sections \ref{sec:kernelsym} and \ref{sec:disc}.

\begin{definition}
For variables $c_i$ and a partition $\lambda=(\lambda_1 \geq \lambda_2 \geq \ldots \geq \lambda_r)$ let
\[
s_\lambda(c) = \det( c_{\lambda_i+j-i} )_{i,j=1,\ldots,r}
\]
be the Schur polynomial. By convention $c_0=1$ and $c_{<0}=0$.
\end{definition}

\begin{example}\rm
{\em The Giambelli-Thom-Porteous formula}. Fix $r\leq n$, $\ell\geq 0$ and let $\Omega^r \subseteq \Hom(\C^n,\C^{n+\ell})$ be the space of linear maps having an $r$-dimensional kernel. It is invariant under the group $GL_n(\C)\times GL_{n+\ell}(\C)$ acting by $(A,B)\cdot \phi=B\circ \phi \circ A^{-1}$. One has \cite{porteous}
\[
[\overline{\Omega^r}]=s_{\lambda}(c),
\]
where
\[
 \lambda=(\underbrace{r+\ell,\ldots, r+\ell}_{r}), \qquad\qquad
1+c_1t+c_2t^2+\ldots=\frac{1+b_1t+b_2t^2+\ldots+b_{n+\ell}t^{n+\ell}}{1+a_1t+a_2t^2+\ldots+a_nt^n}.
\]
Here $a_i$ (respectively $b_i$) is the $i$th universal Chern class of $GL_n(\C)$ (respectively $GL_{n+\ell}(\C)$).
\end{example}

\begin{example} \rm \label{ex:symmetric}
{\em Symmetric 2-forms}. Let $r\leq n$ and let $\Sigma^r=\Sigma^r_n \subseteq \mbox{Sym}^2(\C^n)$ be the collection of symmetric 2-forms having a kernel of dimension $r$. It is invariant under the group $GL_n(\C)$ acting by $A\cdot M=AMA^T$. One has \cite{jlp, pp, ht} that
\[
[\overline{\Sigma}^r_n]=2^r  s_{(r,r-1,\ldots,2,1)}(c),
\]
where $c_i$ is the $i$th universal Chern class of $GL_n(\C)$.
\end{example}

\section{Affine, projective, and restricted projective fundamental classes}\label{sec:projVSaffine}
In this section we recall the formalism of comparing equivariant fundamental classes in affine and projective spaces.

Consider the representation of the torus $T=(\C^*)^k$ acting by
\[
(a_1,\ldots,a_k)\cdot (x_1,\ldots,x_n)=(\prod_{i=1}^k a_i^{s_{1,i}} x_1, \prod_{i=1}^k a_i^{s_{2,i}} x_2, \ldots, \prod_{i=1}^k a_i^{s_{n,i}} x_n).
\]
We will assume that the representation ``contains the scalars'', that is, there exist integers $r_1,\ldots,r_k$ and $r$ such that
\[
\sum_{i=1}^k r_i s_{j,i} = r, \qquad \text{for all} \qquad j=1,\ldots,n.
\]
In other words, the action of $(b^{r_1},\ldots , b^{r_k}) \in T$ ($b\in \C^*$) on $\C^n$ is multiplication by $b^r$.

Under this assumption we have that the non-zero orbits of the linear representation, and the orbits of the induced action on $\PP^{n-1}$ are in bijection. We will compare the ($T$-equivariant) fundamental class of an invariant subvariety $\Sigma\subseteq \C^n$ with the ($T$-equivariant) fundamental class of the projectivization $\PP(\Sigma)\subseteq \PP^{n-1}$. For this we need some notation.

The fundamental class $[\Sigma]$ of $\Sigma$ is an element of $H^*_T(\C^n)=H^*(BT)=\C[\alpha_1,\ldots, \alpha_k]$, where $\alpha_i$ is the equivariant first Chern class of the $\C^*$-action corresponding to the $i$th factor. Hence we can consider  $[\Sigma]$ as a polynomial in the $\alpha_i$'s.

\vskip 3pt

Let $w_j=\sum_{i=1}^k s_{j,i}\alpha_i$, $j=1,\ldots,n$ be the weights of the representation above. Then we have
\[
H^*_T(\PP^{n-1})=H^*(BT)[\xi]/\prod_{j=1}^n (\xi-w_j),
\]
where $\xi$ is the first Chern class of the tautological line bundle over $\PP^{n-1}$.

\begin{proposition} \cite[Thm. 6.1]{fnr1} \label{prop:projTp}
Let $\Sigma$ be a $T$-invariant subvariety of $\C^n$. For the $T$-equivariant fundamental class of $\PP(\Sigma)$ we have
\[
[\PP(\Sigma)]=[\Sigma]|_{\alpha_i \mapsto \alpha_i-\frac{r_i}{r} \xi} \qquad \in H^*_T(\PP^{n-1}).
\]
\end{proposition}

Here, and in the future, by $p(\alpha_i)|_{\alpha_i\mapsto \beta_i}$ we mean the substitution of $\beta_i$ into the variables $\alpha_i$ of the polynomial $p(\alpha_i)$.

We shall need a further twist on this notion. Let $F_j$ be the $j$th coordinate line of $\C^n$, which is a fixed point of the $T$-action on $\PP^{n-1}$. We have the restriction map $H^*_T(\PP^{n-1}) \to H^*_T(F_j)=H^*(BT)$, which we denote by $p \mapsto p|_{F_j}$.

\begin{corollary} \label{cor:projTpCOR}
We have
\[
[\PP(\Sigma)]|_{F_j} = [\Sigma]|_{\alpha_i \mapsto \alpha_i-\frac{r_i}{r} w_j} \qquad \in H^*(BT).
\]
\end{corollary}

\begin{proof}
The restriction homomorphism $H^*_T(\PP^{n-1}) \to H^*_T(F_j)$ is given by substituting $w_j$ for $\xi$.
\end{proof}

\begin{example} \rm
Let $(\C^*)^3$ act on $\C^2$ by $(a_1,a_2,a_3)\cdot (x_1,x_2)=(a_1^3a_2^{-1}a_3\cdot x_1, a_1a_2^2a_3^2 \cdot x_2)$. The numbers $r_1=2, r_2=1, r_3=1, r=6$ prove that this action contains the scalars. Let $\Sigma$ be the $x_1$-axis. Then $[\Sigma]$ is the normal Euler class, that is $[\Sigma]=\alpha_1+2\alpha_2+2\alpha_3$. According to Proposition~\ref{prop:projTp} we have
\[
[\PP(\Sigma)]=\alpha_1+2\alpha_2+2\alpha_3|_{\alpha_1\mapsto \alpha_1-\frac{1}{3}\xi, \alpha_2\mapsto \alpha_2-\frac{1}{6}\xi, \alpha_3\mapsto \alpha_3-\frac{1}{6}\xi}=\alpha_1+2\alpha_2+2\alpha_3-\xi.
\]
According to Corollary \ref{cor:projTpCOR} the two fixed point restrictions of this class are
\begin{align*}
[\PP(\Sigma)]|_{(1:0)} & =\alpha_1+2\alpha_2+2\alpha_3|_{\alpha_1\mapsto \alpha_1-\frac{1}{3}(3\alpha_1-\alpha_2+\alpha_3), \alpha_2\mapsto \alpha_2-\frac{1}{6}(3\alpha_1-\alpha_2+\alpha_3), \alpha_3\mapsto \alpha_3-\frac{1}{6}(3\alpha_1-\alpha_2+\alpha_3)}\\
& =-2\alpha_1+3\alpha_2+\alpha_3
\end{align*}
and
\[
[\PP(\Sigma)]|_{(0:1)}=\alpha_1+2\alpha_2+2\alpha_3|_{\alpha_1\mapsto \alpha_1-\frac{1}{3}(\alpha_1+2\alpha_2+2\alpha_3), \alpha_2\mapsto \alpha_2-\frac{1}{6}(\alpha_1+2\alpha_2+2\alpha_3), \alpha_3\mapsto \alpha_3-\frac{1}{6}(\alpha_1+2\alpha_2+2\alpha_3)}=0.
\]
The vanishing of the second one is expected since the $x_2$-axis is not in $\Sigma$, and the first one can be verified by seeing that the action on $\PP^1$ in the coordinate $t=x_2/x_1$ is
\[
(a_1,a_2,a_3).t=\frac{a_1a_2^{2}a_3^2}{a_1^3a_2^{-1}a_3}\cdot  t=a_1^{-2}a_2^3a_3 \cdot t.
\]
The calculations of this example were deceivingly simple caused by the fact that $\Sigma$ was smooth.
\end{example}

\section{Loci characterised by singular vectors in the kernel} \label{sec:kernelsym}

\subsection{The $\Sigma^r_{e,f}$ locus}

For positive integers $e,f$, let $E:=\C^e$ and $F:=\C^f$ be the standard representations of $GL_e(\C)$ and $GL_f(\C)$ respectively. Consider the induced action of $G=GL_e(\C) \times GL_f(\C)$ on $\Hom(\mbox{Sym}^2 E,F)$. Define the locus
\[
\Sigma_{e,f}^r=
\Bigl\{ \phi \in \Hom(\mbox{Sym}^2 E,F): \exists \ q\in \mbox{Sym}^2 E \text{ with } \dim(\mbox{Ker } q) =r \ \mbox{ and } \  \phi(q)=0 \Bigr\}.
\]
which is invariant under the $G$-action. Using the notation of Example \ref{ex:symmetric} we have
\[
\Sigma_{e,f}^r=
\Bigl\{ \phi \in \Hom(\mbox{Sym}^2 E,F): \exists \ 0\not= q \in \Sigma_e^r \cap \mbox{Ker}(\phi)  \Bigr\}.
\]

We will assume that $d:=\binom{e+1}{2}-f$ is positive, that is, the condition above is not that $\phi$ {\em has} a kernel, but rather that this kernel intersects $\Sigma^r_e \subseteq \mbox{Sym}^2 E$. We shall also assume that this intersection is generically at most $0$-dimensional, that is, $d\leq \codim(\Sigma^r_e \subseteq \mbox{Sym}^2E)= \binom{r+1}{2}$.

\vskip 3pt

In this section our goal is to find a formula for the $G$-equivariant fundamental class
\[
[\overline{\Sigma_{e,f}^r}]\in H^*_G\Bigl(\Hom(\mbox{Sym}^2 E,F)\Bigr)=\C[\alpha_1,\ldots,\alpha_e,\beta_1,\ldots,\beta_f]^{S_e \times S_f}.
\]
Here $\alpha_i$ are the Chern roots of $GL_e(\C)$ (that is, their elementary symmetric polynomials are the Chern classes), and $\beta_i$ are the Chern roots of $GL_f(\C)$
respectively.

\vskip 4pt

The calculation---which will complete the proof of Theorem \ref{classdiv1}---is done via torus-equivariant localization. To bypass complications caused by a complete resolution of $\overline{\Sigma_{e,f}^r}$ we will use a method of \cite{bsz, ts} which requires only a partial desingularization exhibited as a vector bundle over a compact space.

\subsection{A partial resolution}

Let $\F$ be the partial flag manifold parametrizing chains of linear subspaces $C \subseteq D \subseteq \mbox{Sym}^2 E$, where $\dim C=1$ and $\dim D=d$.
Recall that in Example \ref{ex:symmetric} we defined the subset $\Sigma^r=\Sigma^r_e\subseteq \mbox{Sym}^2 E$.
Define
\[
I=\Bigl\{ \bigl((C,D),\phi\bigr) \in \F \times \Hom(\mbox{Sym}^2 E,F) : C\subseteq \overline{\Sigma^r} \ \mbox{ and } \ \phi|_D=0\Bigr\} \ \mbox{ and }
\]
\[
Y=\{ (C,D) \in \F  : C\subseteq \overline{\Sigma^r}  \}
\]
and let $p:I \to Y$ be the map forgetting $\phi$. We have the commutative diagram
\begin{equation}\label{diagram}
\xymatrix{
 I \ar[d]_p \ar@{^{(}->}[r]_{i\ \ \ \ \ \ \ \ \ \ \ \ \ \ } \ar@/^2pc/[rr]^\rho & \F \times \Hom(\mbox{Sym}^2E,F) \ar[d]_{\pi_1} \ar[r]_{\ \ \ \pi_2}& \Hom(\mbox{Sym}^2E,F) \\
Y  \ar@{^{(}->}[r]_j & \F,   &
}
\end{equation}
with $i$ and $j$ being natural inclusions and  $\pi_1$, $\pi_2$ natural projections. The map $\rho=\pi_2 \circ i$ is birational to $\overline{\Sigma_{e,f}^r}$. We have
\[
\dim Y = \binom{e+1}{2} - \binom{r+1}{2} -1 +(d-1)f, \qquad
\dim I = \binom{e+1}{2} - \binom{r+1}{2} -1 +(d-1)f + f^2.
\]
Hence the codimension
\[
\codim \Bigl(  \overline{\Sigma^r_{e,f}} \subseteq \Hom(\mbox{Sym}^2E,F) \Bigr) = \binom{r+1}{2} - \binom{e+1}{2} + f + 1
=\binom{r+1}{2}-d+1,
\]
which is thus the degree of the fundamental class $[\overline{\Sigma^r_{e,f}}]$ we are looking for.

\subsection{Localization and residue formulas}
Let $W=\{\alpha_i+\alpha_j\}_{1\leq i\leq j \leq e}$ be the set of weights of $\mbox{Sym}^2E$. Let $h_r(\alpha_1,\ldots,\alpha_e)$ be the polynomial $2^r s_{(r,r-1,\ldots,1)}(c)$, where $1+c_1t+c_2t^2+\ldots=\prod_{i=1}^e (1+\alpha_it)$ (cf. Example \ref{ex:symmetric}).

\begin{theorem}
Using the notations and assumption above we have
\begin{equation}\label{eq:locformula}
[\overline{\Sigma_{e,f}^r}]=
\mathop{\sum_{H \subseteq W}}_{|H|=d}\sum_{\gamma\in H}
\frac{ h_r|_{\alpha_i\mapsto \alpha_i-\gamma/2} \cdot \displaystyle\prod_{j=1}^f \prod_{\delta\in H} (\beta_j-\delta)}
{\displaystyle\prod_{\delta\in W-\{\gamma\}} (\delta-\gamma) \cdot \prod_{\delta\in H-\{\gamma\}} \prod_{\epsilon\in W-H} (\epsilon-\delta)}.
\end{equation}
\end{theorem}

\begin{proof}
To calculate the fundamental class $[\overline{\Sigma^r_{e,f}}]$ it would be optimal to find an equivariant resolution $\tilde{\Sigma}\to \Hom(\mbox{Sym}^2E,F)$ of $\overline{\Sigma^r_{e,f}}\subseteq \Hom(\mbox{Sym}^2E,F)$, with a well understood Gysin map formula. While the description of such a full resolution is difficult, in diagram (\ref{diagram}) we constructed an equivariant partial resolution $\rho:I\to\Hom(\mbox{Sym}^2E,F)$ of the locus $\overline{\Sigma^r_{e,f}}\subseteq \Hom(\mbox{Sym}^2E,F)$. Although $\rho$ is only a partial resolution (since $I$ is not smooth), it is of special form: $I$ is a {\em vector bundle} over a (possibly singular) subvariety of a {\em compact} space $\F$.

In \cite[Section 3.2]{bsz} and \cite[Section 5]{ts} it is shown that such a partial resolution reduces the problem of calculating $[\overline{\Sigma^r_{e,f}}]$ to calculating the fundamental class $[Y\subseteq \F]$ near the fixed points of the maximal torus. Namely, \cite[Proposition 3.2]{bsz}, or equivalently \cite[Proposition 5.1]{ts}, applied to diagram (\ref{diagram}) gives
\begin{equation}\label{singularfiber}
[\overline{\Sigma_{e,f}^r}]=
\sum_{q}
\frac{ [ Y \subseteq \F ]|_q \cdot [I_q \subseteq \Hom(\mbox{Sym}^2E,F)]}
{e(T_q\F)},
\end{equation}
where $q$ runs through the finitely many torus fixed points of $\F$ and $I_q=p^{-1}(q)$.

Let us start with the obvious ingredients of this formula. The fixed points of $\F$ are pairs $(C, D)$ where $C \subseteq D$ are coordinate subspaces of $\mbox{Sym}^2E$ of dimension 1 and $d$ respectively. The coordinate lines of $\mbox{Sym}^2E$ are in bijection with $W$, and hence the fixed points $q$ are parameterized by choices $H \subset W$ ($|H|=d$) and $\gamma\in H$.
Denoting the tautological rank 1 and rank $d$ bundles over $\F$ by $\mathcal L$ and $\mathcal D$ we have
\[
T\F=\Hom(\mathcal L,\mathcal D/\mathcal L) \oplus \Hom(\mathcal L, \mbox{Sym}^2E/\mathcal D) \oplus \Hom(\mathcal D/\mathcal L, \mbox{Sym}^2E/\mathcal D).
\]
Hence, for a fixed point $q$ corresponding to $(H,\gamma)$ we have
\begin{itemize}
\item
$[I_q \subseteq \Hom(\mbox{Sym}^2E,F)]
=
\prod_{j=1}^f \prod_{\delta\in H} (\beta_j-\delta)$,
\item
$e(T_q\F)=\prod_{\delta\in W-\{\gamma\}} (\delta-\gamma) \cdot \prod_{\delta\in H-\{\gamma\}} \prod_{\epsilon\in W-H} (\epsilon-\delta)$,
\end{itemize}
both following from the fact that for a $G$-representation $K$ and invariant subspace $L\subseteq K$ the fundamental class $[L\subset K]$ is the product of the weights of $K/L$.

It remains to find the non-obvious ingredient of formula (\ref{singularfiber}), the local fundamental class $[ Y \subseteq \F ]|_q$. However, this problem was essentially solved in Section~\ref{sec:projVSaffine}. The space $Y$ is the complete preimage of $\PP(\overline{\Sigma^r})$ under the fibration $z:\F \to \PP(\mbox{Sym}^2 E)$. Hence $[ Y \subseteq \F ]|_q=[\PP(\overline{\Sigma^r})]|_{z(q)}$. We have $[\overline{\Sigma^r}]=h_r(\alpha_1,\ldots,\alpha_e)$ (see Example \ref{ex:symmetric}), and hence Corollary~\ref{cor:projTpCOR} calculates $[\PP(\overline{\Sigma^r})]|_{z(q)}$ to be $h_r|_{\alpha_i\mapsto \alpha_i-\gamma/2}$. This completes the proof.
\end{proof}

\begin{example} \rm
We have
\[
[\overline{\Sigma^1_{2,2}}]=
\frac{ (\beta_1-2\alpha_1)(\beta_2-2\alpha_1)}{\alpha_2-\alpha_1}+
\frac{ (\beta_1-2\alpha_2)(\beta_2-2\alpha_2)}{\alpha_1-\alpha_2}=
-4(\alpha_1+\alpha_2)+2(\beta_1+\beta_2).
\]
\end{example}

More structure of the localization formula (\ref{eq:locformula}) will be visible if we rewrite it as a residue formula, with the help of the following lemma, which we prepare by setting some notation.

\vskip 4pt

Let $0\leq k_1\leq k_2\leq \ldots \leq k_r$ be integers and let $V$ be a vector bundle of rank $k_r$ on $X$. Let $p:\F_{k_1,\ldots,k_r}(V)\to X$ be the bundle whose fiber over $x\in X$ is the variety of chains of linear subspaces
$
V_{1}^{k_1}\subseteq V_{2}^{k_2} \subseteq \ldots \subseteq V_{r}^{k_r} =V_x,
$
where upper indices indicate dimension and $V_x$ is the fiber of $V$ over $x$. The Chern roots of the tautological bundle of rank $k_i$ over $\F_{k_1,\ldots,k_r}(V)$ will be denoted by $\sigma_{i,j}$ for $i=1,\ldots,r$ and $j=1,\ldots,k_i$. The $\sigma_{r,j}$ classes are the pullbacks of the Chern roots of $V$. In notation we do not indicate the pullback, so $\sigma_{r,j}$ will also denote the Chern roots of $V$.

\begin{lemma}\label{lem:pf}
Consider the variables $z_{i,j}$ for $i=1,\ldots,r-1$, $j=1,\ldots,k_i$, and let $z_{r,j}=\sigma_{r,j}$.
Let $g(z_{i,j})$ be a polynomial symmetric in the sets of variables $z_{i*}$ for all $i$, and let $D=\sum_{i<j} (k_i-k_{i-1})(k_j-k_{j-1})$ be the dimension of the fiber of $p$.  We have
\begin{equation}\label{eqn:intpf}
p_*(g(\sigma_{i,j}))=
(-1)^{D}
\left\{
\frac{ g(z_{i,j}) \prod_{i=1}^{r-1} \prod_{1\leq u<v\leq k_i} \bigl(1-\frac{z_{i,u}}{z_{i,v}}\bigr)}
{\prod_{i=1}^{r-1}\prod_{j=1}^{k_i} z_{i,j}^{k_{i+1}-k_i} \cdot \prod_{i=1}^{r-1} \prod_{u=1}^{k_{r+1}} \prod_{v=1}^{k_r} \bigl(1-\frac{z_{i+1,u}}{z_{i,v}}\bigr)  }
\right\}_{z_{1*}^0\ldots z_{k-1,*}^0},
\end{equation}
where, by $\{P \}_{z_{1*}^0\ldots z_{k-1,*}^0}$ we mean the constant term in the variables $z_{i,j}$ for $i=1,\ldots,k-1$ and $j=1,\ldots,k_i$, of the Laurent expansion of $P$ in the region $|z_{1,j_1}|>|z_{2,j_1}|>\ldots>|z_{r,j_r}|$.
\end{lemma}

\begin{proof} First we prove the statement for $r=2$. To that end, we temporarily rename $k_1=k$, $k_2=n, \sigma_{1,j}=\sigma_j$, $\sigma_{2,j}=\tau_j, z_{1,j}=z_j$, and we shall use the abbreviations $\sigma=(\sigma_1,\ldots,\sigma_k)$, $\tau=(\tau_1,\ldots,\tau_n)$, $z=(z_1,\ldots,z_k)$. By \cite[Lemma 2.5]{pr88} 
we have
\begin{equation}\label{eqn:pf2}
p_*(g(\sigma,\tau))=
\sum_I
\frac{g(\tau_I,\tau)}{\prod_{j\not\in I}\prod_{i\in I} (\tau_j-\tau_i)}
\end{equation}
where the summation is over $k$-element subsets $I=\{s_1,\ldots,s_k\}$ of $\{1,\ldots,n\}$ and $\tau_I=(\tau_{s_1},\ldots,\tau_{s_k})$. Define
\[H= (-1)^{k(n-k)} g(z,\tau)\prod_{1\leq i<j\leq k}(z_j-z_i)  \cdot  \frac{z_1^{k-1}z_2^{k-2}\ldots z_{k-1}}{\prod_{j=1}^n\prod_{i=1}^k (z_i-\tau_j)}\]
and consider the differential form
$\omega=H dz_1\wedge \ldots \wedge dz_k$.
Let $R=\Res_{z_k=\infty} \Res_{z_{k-1}=\infty}\ldots \Res_{z_1=\infty}(\omega)$.

First we calculate R by applying the Residue Theorem (the sum of the residues of a meromorphic form on the Riemann sphere is 0) for $z_1, z_2, \ldots,z_k$. We obtain
\[
R=(-1)^k\sum_{s_k}\sum_{s_{k-1}}\ldots \sum_{s_1}\Res_{z_k=\tau_{s_k}} \Res_{z_{k-1}=\tau_{s_{k-1}} }\ldots \Res_{z_1=\tau_{s_1}}(\omega).
\]
The terms corresponding to choices with non-distinct $s_j$'s is 0, due to the factor $\prod (z_j-z_i)$ in the numerator of $\omega$.  Thus we have
\[
R=(-1)^{k(n-k)+k}\sum_{I} \sum_{w\in S_k} \frac{ g(\tau_I,\tau) \prod_{i<j} (\tau_{w(s_j)}-\tau_{w(s_i)}) \tau_{w(s_1)}^{k-1}\tau_{w(s_2)}^{k-2}\ldots \tau_{w(s_{k-1})}}
{\prod_{i\not= j}(\tau_{w(s_j)}-\tau_{w(s_i)}) \prod_{j\not\in I}\prod_{i=1}^k  (\tau_{w(s_i)}-\tau_j)},
\]
where the summation is over $k$-element subsets $I=\{s_1,\ldots,s_k\}\subset \{1,\ldots,n\}$. This further equals
\[
R=(-1)^k\sum_I \left(
\frac{g(\tau_I,\tau)}{\prod_{j\not\in I}\prod_{i\in I} (\tau_j-\tau_i)}
\underbrace{\sum_{w\in S_k}  \frac{  \tau_{w(s_1)}^{k-1}\tau_{w(s_2)}^{k-2}\ldots \tau_{w(s_{k-1})}}{\prod_{i>j} ( \tau_{w(s_j)}-\tau_{w(s_i)})}}_{(*)}
\right).
\]
However, the sum marked by (*) is equal to 1---because of the well known product form of a Vandermonde determinant---, and using (\ref{eqn:pf2}) we obtain that
$p_*(g(\sigma,\tau))=(-1)^k R$. Calculating the residues at infinity as a coefficient of the Laurent expansion we get
\[
p_*(g(\sigma,\tau)) = (-1)^k R = \left\{H \cdot \prod_{i=1}^k z_i \right\}_{z_1^0\ldots z_k^0},
\]
where $\{\ \}_{z_1^0\ldots z_k^0}$ means the constant term of the Laurent-expansion in the $|z_i|>|\tau_j|$ (for all $i,j$) region. This proves (\ref{eqn:intpf}) for $r=2$.

For $r>2$ the push-forward map $p_*$ can be factored as $p_{1*}\circ p_{2*}\circ \ldots \circ p_{r*}$ for the Grassmanian fibrations
\[
p_i: \F_{k_i,k_{i+1},\ldots,k_r}(V) \to \F_{k_{i+1},\ldots,k_r}(V),
\]
with the notation $\F_{\emptyset}(V)=X$. The map $p_i$ is a special case of the construction in the theorem for $r=2$ and the tautological rank $k_{i+1}$ bundle over $\F_{k_{i+1},\ldots,k_r}(V)$. Hence $p_{i*}$ can be computed with the formula in the theorem (as it is proved for $r=2$ above). The iterated application of (\ref{eqn:intpf}) for $r=2$ gives the general (\ref{eqn:intpf}), which completes the proof of the theorem.
\end{proof}

\begin{theorem} \label{thm:residue}
We have
\[
[\overline{\Sigma_{e,f}^r}]=(-1)^{d+1}
\left\{
\frac{ h_r|_{\alpha_i\mapsto \alpha_i-z/2} \cdot \prod_{1\leq i<j\leq d}(1-\frac{u_i}{u_j})}
{z^{d-1}\prod_{j=1}^d (1-\frac{u_j}{z})}
\cdot
\prod_{j=1}^d  \sum_{i=0}^\infty \frac{c_i(F^{\vee}-\mathrm{Sym}^2E^{\vee})}{u_j^i}
\right\}_{z^0u^0},
\]
where $\{P\}_{z^0u^0}$ means the constant term in $P$ with respect to $z$ and $u_1,\ldots,u_d$.
\end{theorem}

\begin{proof}
The formula (\ref{singularfiber}) for $[\overline{\Sigma_{e,f}^r}]$ is the Atiyah-Bott localization formula for the equivariant push-forward
$p_*( [Y\subseteq \F] e(\Hom(\mathcal{D},F))$,
where $\mathcal{D}$ is the tautological rank $d$ bundle over $\F$, and $p:\F\to$pt.
Calculating the equivariant push-forward $p_*$ with the formula in Lemma \ref{lem:pf}, we obtain
\begin{equation}\label{eqn:tt}
(-1)^{d\binom{e+1}{2}-d^2+d-1}
\left\{
\frac{ h_r|_{\alpha_i\mapsto \alpha_i-z/2} \prod_{i=1}^f\prod_{j=1}^d (\beta_i-u_j) \prod_{1\leq i<j\leq d}(1-\frac{u_i}{u_j})}
{z^{d-1} (u_1\cdots u_d)^{\binom{e+1}{2}-d}\prod_{j=1}^d (1-\frac{u_j}{z})
\prod_{j=1}^d\prod_{\epsilon \in W} (1-\frac{\epsilon}{u_j})}
\right\}_{z^0u^0}.
\end{equation}
Observing that
\begin{align*}
\prod_{j=1}^d \frac{\prod_{i=1}^f (\beta_i-u_j)}{\prod_{\epsilon\in W} (1-\epsilon/u_j)}& =
 (-1)^{df}\prod_{j=1}^d u_j^f \prod_{j=1}^d \frac{\prod_{i=1}^f (1-\beta_i/u_j)}{\prod_{\epsilon\in W}(1-\epsilon/u_j)} \\
 &=
 (-1)^{df}\prod_{j=1}^d u_j^f \sum_{i=0}^\infty \frac{ c_i(F^{\vee}-\mbox{Sym}^2 E^{\vee})}{u_j^i},\\
\end{align*}
 and that $f=\binom{e+1}{2}-d$, we have that (\ref{eqn:tt}) further equals the formula in the theorem.
\end{proof}

\subsection{The divisorial case}
The residue formula of Theorem \ref{thm:residue} is more manageable in case the codimension of $\Sigma^r_{e,f}$ is 1---the case relevant for most applications given in this paper. After two technical lemmas we will provide a simple formula for the $[\overline{\Sigma^r_{e,f}}]$ in this case.

\begin{lemma}\label{lem:tech1}
For the $z$-expansion of the polynomial $h_r|_{\alpha_i \mapsto \alpha_i-z/2}$ we have
\begin{equation}\label{eqn:expan}
h_r|_{\alpha_i \mapsto \alpha_i-z/2}=
(-1)^{\binom{r+1}{2}}
\left(
A^r_e z^{ \binom{r+1}{2} } +  B^r_e \cdot \sum_{i=1}^e \alpha_i \cdot z^{ \binom{r+1}{2} -1} +\mathrm{l.o.t.}
\right)
\end{equation}
where
\begin{align*}
A_e^r=&2^{-\binom{r}{2}} \det\begin{pmatrix} \binom{e}{r+1-2i+j} \end{pmatrix}_{i,j=1,\ldots,r}
=
\frac{{e\choose r}{e+1\choose r-1}\cdots {e+r-1\choose 1}}{{1\choose 0}{3\choose 1}{5\choose 2}\cdots {2r-1\choose r-1}},
\\
B_e^r=&-\frac{2}{e}\binom{r+1}{2}A^r_e.
\end{align*}
\end{lemma}

\begin{proof} The polynomial $h_r$ is a homogeneous degree $\binom{r+1}{2}$ symmetric polynomial in the $\alpha_1,\ldots, \alpha_e$ variables. Hence the expansion (\ref{eqn:expan}) must hold for some numbers $A^r_e, B^r_e$. We will calculate them via the substitution $\alpha_1=\cdots=\alpha_e$. Let $D=\det\begin{pmatrix} \binom{e}{r+1-2i+j} \end{pmatrix}_{i,j=1,\ldots,r}$. From the definition of $h_r$ we see that
$
h_r(\underbrace{\alpha,\ldots,\alpha}_{e})
=
2^rD\alpha^{\binom{r+1}{2}}
$,
and hence, for the $z$-expansion of  $h_r(\alpha-\frac{z}{2},\ldots,\alpha-\frac{z}{2})$ we get
\[
2^r D \left(-\frac12\right)^{\binom{r+1}{2}} z^{\binom{r+1}{2}}+
 2^r D\binom{r+1}{2} \left(-\frac12\right)^{\binom{r+1}{2}-1}  \frac{1}{e} (e\alpha) z^{\binom{r+1}{2}-1}+ \mathrm{l.o.t.},
\]
which proves the first expression for $A^r_e$ and the expression for $B^r_e$. The equivalence of the two displayed expressions for $A_e^r$ is proved in \cite[Proposition 12]{ht}.
\end{proof}

\begin{lemma}\label{lem:tech2}
We have
\begin{equation}\label{eqn:pl}
\prod_{1\leq i<j\leq d}\left(1-\frac{u_i}{u_j}\right)=1-\sum_{i=1}^{d-1} \frac{u_i}{u_{i+1}}+Q,
\end{equation}
where $Q$ is the sum of $u$-monomials in which the degree of the denominator is at least two. Also,
\[
\left(\sum_{i=1}^d u_i\right) \cdot
\prod_{1\leq i<j\leq d}\left(1-\frac{u_i}{u_j}\right)=u_d+\mathrm{fractions},
\]
where {\em fractions} stands for terms of monomials with at least one $u_i$ in the denominator.
\end{lemma}
For example, if $d=3$ then we have
\[\left(1-\frac{u_1}{u_2}\right)\left(1-\frac{u_1}{u_3}\right)\left(1-\frac{u_2}{u_3}\right)=
1-\frac{u_1}{u_2}-\frac{u_2}{u_3}+ \underbrace{ \left(\frac{u_1u_2}{u_3^2}+ \frac{u_1^2}{u_2u_3}- \frac{u_1^2}{u_3^2}\right)}_{Q},
\]
and $(u_1+u_2+u_3)\prod_{i<j\leq 3}(1-u_i/u_j)=u_3+$fractions.

\begin{proof}
Arguing by induction on $d$ we have that the left hand side of (\ref{eqn:pl}) is
\begin{multline*}
\left(1-\sum_{i=1}^{d-2} \frac{u_i}{u_{i+1}}+Q'\right)\prod_{i=1}^{d-1} \left(1-\frac{u_i}{u_d}\right)
=
\left(1-\sum_{i=1}^{d-2} \frac{u_i}{u_{i+1}}+Q_1\right)\left(1-\sum_{i=1}^{d-1} \frac{u_i}{u_d}+Q_2\right)=
\\
1-\sum_{i=1}^{d-2} \frac{u_i}{u_{i+1}} -\sum_{i=1}^{d-1} \frac{u_i}{u_d} + \sum_{i=1}^{d-2} \frac{u_i}{u_d}+Q
=
1-\sum_{i=1}^{d-1} \frac{u_i}{u_{i+1}}+Q,
\end{multline*}
where $Q_1$ and $Q_2$ are sums of terms that multiplied with anything in the other factor will result in monomials with denominator degree at least 2.

The second statement of the lemma follows directly from the first one.
\end{proof}

We now determine the class of $\overline{\Sigma^r_{e,f}}$ when it is a divisor, which leads
to a proof of Theorem \ref{classdiv1}.

\begin{theorem}\label{divclass}
Assume that $\overline{\Sigma^r_{e,f}}$ is a divisor, that is,
\begin{equation} \label{eqn:divisor_condition}
 \binom{r+1}{2} - d +1 =\binom{r+1}{2} - \binom{e+1}{2} + f+1 =1.
\end{equation}
Then
\begin{equation} \label{eqn:divisor_class}
[\overline{\Sigma_{e,f}^r}]=
A_e^r\left( c_1(F)- \frac{2f}{e} c_1(E)\right).
\end{equation}
\end{theorem}

\begin{proof}
Under the assumption (\ref{eqn:divisor_condition}) Theorem \ref{thm:residue} reads
\[
[\overline{\Sigma_{e,f}^r}]=-
\left\{
\left( A^r_e z^{1 } +  B^r_e \cdot \sum_{i=1}^e \alpha_i \cdot z^{ 0} +\mbox{l.o.t.}\right)
\cdot
\frac{ \prod_{1\leq i<j\leq d}\Bigl(1-\frac{u_i}{u_j}\Bigr)}
{\prod_{j=1}^d \Bigl(1-\frac{u_j}{z}\Bigr)}
\cdot
\prod_{j=1}^d  \sum_{i=0}^\infty \frac{c_i(F^{\vee}-\mbox{Sym}^2E^{\vee})}{u_j^i}
\right\}_{z^0,u^0}.
\]
Looking at the $z$-exponents, this is further equal to
\[
-\left\{
\left(
A^r_e \sum_{j=1}^d u_j \prod_{1\leq i<j\leq d}\Bigl(1-\frac{u_i}{u_j}\Bigr) +
B_e^r \sum_{i=1}^e \alpha_i \prod_{1\leq i<j\leq d}\Bigl(1-\frac{u_i}{u_j}\Bigr)
\right)
\cdot
\prod_{j=1}^d  \sum_{i=0}^\infty \frac{c_i(F^{\vee}-\mbox{Sym}^2E^{\vee})}{u_j^i}
\right\}_{u^0}.
\]
Looking at $u$-exponents, and using Lemma \ref{lem:tech2}, this is further equal to
\[
-\left\{
\left(
A^r_e \Bigl(u_d + \text{fractions}\Bigr) +
B_e^r \sum_{i=1}^e \alpha_i (1+\text{fractions})
\right)
\cdot
\prod_{j=1}^d  \sum_{i=0}^\infty \frac{c_i(F^{\vee}-\mbox{Sym}^2E^{\vee})}{u_j^i}
\right\}_{u^0},
\]
where the term \emph{fractions} stands for terms with at least one $u_j$ variable in the denominator. Hence the formula further equals
\[
-A^r_e c_1(F^{\vee}-\mbox{Sym}^2E^{\vee}) - B_e^r c_1(E).
\]
Using that $c_1(F^{\vee}-\mbox{Sym}^2E^{\vee})=c_1(F^{\vee})-c_1(\mbox{Sym}^2E^{\vee})=-c_1(F)+(e+1)c_1(E)$, we obtain
\[
[\overline{\Sigma_{e,f}^r}]=
A_e^r c_1(F)-
\left(
A_e^r (e+1)+B_e^r
\right) c_1(E).
\]
Using the divisorial condition (\ref{eqn:divisor_condition}), this expression can be rewritten as (\ref{eqn:divisor_class}).
\end{proof}

\begin{example} \rm
We have
\[
[\overline{\Sigma^1_{2,2}}]=-4c_1(E)+2c_1(F), \quad
[\overline{\Sigma^1_{3,5}}]=-10c_1(E)+3c_1(F), \quad
[\overline{\Sigma^1_{4,9}}]=-18c_1(E)+4c_1(F),
\]
\[
[\overline{\Sigma^2_{3,3}}]=-8c_1(E)+4c_1(F), \quad
[\overline{\Sigma^2_{4,7}}]=-35c_1(E)+10c_1(F), \quad
[\overline{\Sigma^2_{5,12}}]=-96c_1(E)+20c_1(F).
\]
\end{example}


\section{Loci defined by discriminant} \label{sec:disc}

Let $e\geq 2$ and use the short hand notation $N=\binom{e+1}{2}-2$. Let $E:=\C^e$ be the standard representations of $GL_e(\C)$. Consider the tautological exact sequence of $GL_e(\C)$-equivariant bundles $0 \to S \to S^2E \to Q \to 0$ over the Grassmannian $\Gr(2, \mbox{Sym}^2E)$ of $2$-planes in $\mbox{Sym}^2E$. Recall that we have introduced in Example \ref{ex:symmetric} the $GL_e(\C)$-invariant subset $\Sigma^1 \subseteq \mbox{Sym}^2 E$ as the set of {\em degenerate} symmetric 2-forms. Define
\[
\Phi_e := \Bigl\{ W\in \Gr(2, \mbox{Sym}^2 E): \PP(W) \text{ is tangent to } \PP(\Sigma^1)  \Bigr\}
\subseteq \Gr(2, \mbox{Sym}^2 E).
\]
Notice that we require $\PP(W)$ to be tangent to $\PP(\Sigma^1)$ (which is a smooth but not closed subvariety of $\PP(\mbox{Sym}^2 E)$), that is we require that the projective line $\PP(W)$ intersect the smooth part of $\PP(\overline{\Sigma^1})$, and the intersection be tangential.
Our goal in this section is to calculate the equivariant fundamental class $[\overline{\Phi}_e] \in H^2(\Gr(2, \mbox{Sym}^2E))$.

\vskip 4pt

Denote the $GL_e(\C)$-equivariant Chern roots of $S$ by $\gamma_1,\gamma_2$, those of $E$ by $\alpha_1,\ldots,\alpha_e$, and those of $Q$ by $\beta_1,\ldots,\beta_N$. The $GL_e(\C)$-equivariant cohomology ring of $\Gr(2,\mbox{Sym}^2E)$ can be presented by one of
\[
\C[\alpha_1,\ldots,\alpha_e,\gamma_1,\gamma_2]^{S_e\times S_2}/\text{relations}
\qquad \text{or} \qquad
\C[\alpha_1,\ldots,\alpha_e,\beta_1,\ldots,\beta_N]^{S_e\times S_N}/\text{relations}.
\]
Since in each case the relations have degree $>2$, the class $[\overline{\Phi}_e]$ is a well-defined linear polynomial $f(\alpha_1,\ldots,\alpha_e,\gamma_1,\gamma_2)$ in the $\alpha$ and $\gamma$ variables, or a well-defined linear polynomial $g(\alpha_1,\ldots,\alpha_e,$ $\beta_1,\ldots,\beta_N)$ in the $\alpha$ and $\beta$ variables. The exactness of the $0\to S\to \mbox{Sym}^2E \to Q\to 0$ sequence implies
\begin{equation}\label{eqn:exact}
\sum_{i=1}^2 \gamma_i + \sum_{i=1}^N \beta_i = \sum_{1\leq i\leq j \leq e} (\alpha_i+ \alpha_j),
\end{equation}
hence $f(\alpha,\gamma)$ and $g(\alpha,\beta)$ determine each other.

The polynomials $f(\alpha,\gamma)$ and $g(\alpha,\beta)$ have ``degeneracy locus'' interpretations as follows.
\begin{itemize}
\item
Consider the $GL_2(\C) \times GL_e(\C)$ representation $\Hom(\C^2, \mbox{Sym}^2E)$ given by the following action $(A,B)\cdot \phi:=\mbox{Sym}^2B\circ \phi \circ A^{-1}$, and the locus
\[
\Phi'_e:=
\Bigl\{ \phi \in  \Hom(\C^2, \mbox{Sym}^2E) : \rk(\phi)=2 \ \mbox{ and }   \PP(\mbox{Im}(\phi)) \text{ is tangent to } \PP(\Sigma^1)  \Bigr\}.
\]
Then
\begin{align*}
[\overline{\Phi'_e}] = f(\alpha,\gamma) & \in H^*_{GL_2(\C)\times GL_e(\C)}(\Hom(\C^2, \mbox{Sym}^2E))\\
& = \C[\gamma_1,\gamma_2,\alpha_1,\ldots,\alpha_e]^{S_2\times S_e}.
\end{align*}
\item
Consider the $GL_e(\C) \times GL_N(\C)$ representation $\Hom(\mbox{Sym}^2E, \C^N)$ given by the following action $(A,B)\cdot \phi:=B\circ \phi \circ \mbox{Sym}^2A^{-1}$, and the locus
\[
\Phi_e'':=
\Bigl\{ \phi\in \Hom(\mbox{Sym}^2E, \C^N) : \dim \mbox{Ker}(\phi)=2 \ \mbox{ and } \ \PP(\mbox{Ker}(\phi)) \text{ is tangent to } \PP(\Sigma^1)  \Bigr\}.
\]
Then
\begin{align*}
[\overline{\Phi_e''}] = g(\alpha,\beta) & \in H^*_{GL_e(\C)\times GL_N(\C)}\Bigl(\Hom(\mbox{Sym}^2E,\C^N)\Bigr)\\
& = \C[\alpha_1,\ldots,\alpha_e,\beta_1,\ldots,\beta_N]^{S_e\times S_N}.
\end{align*}
\end{itemize}

\begin{theorem} \label{prop:disc1}
We have
\[
f(\alpha,\gamma)=(e-1) \left( 4\sum_{i=1}^n \alpha_i - e \sum_{i=1}^2 \gamma_i \right).
\]
\end{theorem}

\begin{proof}
For $\phi \in \Hom(\C^2,\mbox{Sym}^2E)$ let $\phi((1,0))=K$, $\phi((0,1))=L$. The equation of the hypersurface $\overline{\Phi_e'}$ in terms of the entries of $K$ and $L$ is the \emph{discriminant} of the polynomial  $\det( \lambda K + L)=a_e(K,L)\lambda^e+a_{e-1}(K,L)\lambda^{e-1}+\ldots+a_0(K,L)$.

Consider the Sylvester matrix form of the discriminant
\[
\frac{1}{a_e}\det
\begin{pmatrix}
a_0 & a_1 & a_2 & \cdots &  a_{e-1}   & a_e & \\
      & a_0 & a_1 & & \cdots     & a_{e-1}&a_e & \\
 & &\ddots & \ddots & & & \ddots &\ddots  \\
     &       &         & a_0 & a_1 & & \cdots &  a_{e-1} & a_e\\
a_1 & 2a_2 & \cdots & &  ea_e     & & \\
      & a_1 & 2a_2 & \cdots        &  & ea_{e}&  & \\
 & & \ddots & \ddots & & & \ddots&  \\
 & & & \ddots & \ddots & & &\ddots  \\
     &       &         &  & a_1 & 2a_2  & \cdots &   & ea_e\\
\end{pmatrix}_{((e-1)+e) \times ((e-1)+e)}.
\]
One of the terms of its expansion (the one coming from the main diagonal) is a non-zero constant times $(a_0a_e)^{e-1}$. We have $a_e(K,L)=\det(K)$ and $a_0(K,L)=\det(L)$. Hence one of the monomials appearing in the discriminant is
$(\prod_{i=1}^e K_{ii})^{e-1}(\prod_{i=1}^e L_{ii})^{e-1}$. The weight of this monomial is
\begin{equation} \label{eqn:temp1}
(e-1) \left(  \sum_{i=1}^e (2\alpha_i -\gamma_1) \right)+
(e-1) \left(  \sum_{i=1}^e (2\alpha_i -\gamma_2) \right).
\end{equation}
Since $\overline{\Phi_e'}$ is invariant, all other terms must have the same weight, and this weight is the equivariant fundamental class of $\overline{\Phi_e'}$. Expression (\ref{eqn:temp1}) simplifies to the formula in the theorem.
\end{proof}
\begin{remark}
Instead of the Sylvester matrix we could have used specializations of advanced equivariant formulas for more general discriminants, see for instance \cite{fnr2}.
\end{remark}

\begin{theorem} \label{prop:disc2}
We have
\[
g(\alpha,\beta)=(e-1) \left( e\sum_{i=1}^N \beta_i - (e^2+e-4) \sum_{i=1}^e \alpha_i \right).
\]
\end{theorem}

\begin{proof}
The statement follows from Theorem \ref{prop:disc1} using relation (\ref{eqn:exact}).
\end{proof}

This completes the proof of Theorem \ref{degpencint}.

\section{The Petri divisor on the moduli space of curves}

An immediate application of the Theorem \ref{classdiv1} concerns the calculation of the class of the Petri divisor on $\mm_g$ consisting of genus $g$ curves whose
canonical model lies on a rank $3$ quadric. We fix some notation. For $1\leq i\leq \lfloor \frac{g}{2}\rfloor$, let $\Delta_i\subseteq \mm_g$ be the boundary divisor of $\mm_g$ whose general point is a union of two smooth curves of genera $i$ and $g-i$ meeting in one point. We denote by $\Delta_0$ the closure of the locus of irreducible stable curves of genus $g$. As customary, we set $\delta_i=[\Delta_i]_{\mathbb Q}\in CH^1(\mm_g)$ for $i=0,\ldots, \lfloor \frac{g}{2}\rfloor$ and denote by $$\delta:=\delta_0+\delta_1+\cdots+\delta_{\lfloor \frac{g}{2}\rfloor}$$ the class of the total boundary. Often we work with the partial compactification $\widetilde{\cM}_g:=\cM_g\cup \Delta_0$, for which $CH^1(\widetilde{\cM}_g)=\mathbb Q\langle \lambda, \delta_0\rangle$.

\vskip 3pt
\begin{definition}\label{quadrk} For a projective variety $X$ and a line bundle $L\in \mathrm{Pic}(X)$, for each integer $k\geq 0$ we denote by $I_{X,L}(k):=\mathrm{Ker}\bigl\{\mathrm{Sym}^k H^0(X,L)\rightarrow H^0(X,L^{\otimes k})\bigr\}$ and set $I_{X,L}:=\oplus_{k\geq 0} I_{X,L}(k)$.
\end{definition}

\vskip 5pt

We fix a smooth non-hyperelliptic  curve $C$ of genus $g$. From M. Noether's Theorem \cite{ACGH} the multiplication map $\mbox{Sym}^2 H^0(C,\omega_C)\rightarrow H^0(C,\omega_C^{\otimes 2})$
is surjective. The space $I_C(2)=I_{C,\omega_C}(2)$ of quadrics containing the canonical curve $C\hookrightarrow \PP^{g-1}$ has dimension
$$\mbox{dim } I_C(2)={g-2\choose 2}.$$  We conclude that the locus
$\mathcal{GP}_g$ of curves whose canonical model lies on a rank $3$ quadric
is expected to be a divisor. Via the Base Point Free Pencil Trick \cite{ACGH} p. 126, this expectation can be confirmed.

\begin{proposition}\label{gp1}
The locus $\mathcal{GP}_g$ coincides set-theoretically with the divisor of curves $[C]\in \cM_g$ having a pencil $A$ such that the Petri map
$\mu(A):H^0(C,A)\otimes H^0(C,\omega_C\otimes A^{\vee})\rightarrow H^0(C,\omega_C)$ is not injective.
\end{proposition}
\begin{proof}
Let $A$ be a line bundle on $C$ with $h^0(C,A)=2$.  Denote by $F:=\mbox{bs } |A|$ its base locus and set $B:=A(-F)$. Thus $H^0(C,B)\cong H^0(C,A)$. Applying the Base Point Free Pencil Trick, we obtain
$$\mbox{Ker}(\mu(A))\cong H^0(\omega_C\otimes A^{-2}(F))\cong H^0(C,\omega_C\otimes B^{-2}(-F)).$$ Thus if $\mu(A)$ is not injective, by possibly enlarging the effective divisor $F$, we find there exists a base point free
pencil $B$ on $C$ and an effective divisor $F$, such that $\omega_C=B^{2}(F)$.

\vskip 4pt
Assume the canonical curve $C\subseteq \PP^{g-1}$ lies on a rank $3$ quadric $Q$. Denote by $F:=C\cdot \mathrm{Sing}(Q)$, where $\mathrm{Sing}(Q)\cong \PP^{g-4}$. Then if $B$ is the pull back to $C$ of the unique ruling of $Q$, we obtain the relation $\omega_C=\OO_C(1)\cong B^{2}(F)$. Setting $A:=B(F)$, we obtain that $\mu(A)$ is not injective.

\vskip 3pt
To conclude that $\mathcal{GP}_g$ is a divisor in $\cM_g$, we invoke the Gieseker-Petri Theorem which asserts that the Petri map $\mu(A)$  is \emph{injective} for
\emph{every} line bundle $A$ on a general curve $C$ of genus $g$.
\end{proof}

The divisor $\mathcal{GP}_g$ can  be extended over $\mm_g$. Let $\pi:\mm_{g,1}\rightarrow \mm_g$ the universal curve of genus $g$. We denote by $\mathbb E:=\pi_*(\omega_{\pi})$ the Hodge bundle on $\mm_g$,  having fibres $\mathbb E[C]:=H^0(C,\omega_C)$. Let $\mathbb F:=\pi_*(\omega_{\pi}^{\otimes 2})$. Both sheaves $\mathbb E$ and $\mathbb F$ are locally free over $\mm_g$ and denote by
$$\phi:\mbox{Sym}^2(\mathbb E)\rightarrow \mathbb F$$
the morphism globalizing the multiplication maps $\phi_C: \mbox{Sym}^2 H^0(C,\omega_C)\rightarrow H^0(C,\omega_C^{\otimes 2})$, as the curve $[C]\in \mm_g$ varies in moduli.
Set $$\widetilde{\mathcal{GP}}_g:=\Bigl\{[C]\in \mm_g: \exists \ 0\neq q\in \mathrm{Ker}(\phi_C), \ \ \mbox{rk}(q)\leq 3\Bigr\}.$$
Clearly $\widetilde{\mathcal{GP}}_g$ is a divisor on $\mm_g$ and $\widetilde{\mathcal{GP}}_g\cap \cM_g=\mathcal{GP}_g$.
For a generic point $[C:=C_1\cup_p C_2]\in \Delta_i$, where $C_1$ and $C_2$ are smooth curves of genus $i$ and $g-i$ respectively meeting at one point $p$, one has
$H^0(C,\omega_C)\cong H^0(C_1,\omega_{C_1})\oplus H^0(C_2,\omega_{C_2})$, that is, every section from $H^0(C,\omega_C)$ vanishes at $p$. On the other hand,
$$H^0(C,\omega_C^{ 2})\cong \mbox{Ker}\bigl\{H^0(C_1,\omega_{C_1}^{2}(2p))\oplus H^0(C_2,\omega_{C_2}^{2}(2p))\rightarrow \mathbb C_p\bigr\},$$
that is, there exists quadratic differentials on $C$ \emph{not} vanishing at $p$.
It follows that the multiplication map $\phi_C$ is not surjective, hence for dimension reasons $\mbox{Ker}(\phi_C)$ contains quadrics of rank $3$, whenever $[C]\in \Delta_i$. Thus $\Delta_i\subseteq \widetilde{\mathcal{GP}}_g$, for $i=1,\ldots, \lfloor \frac{g}{2}\rfloor$. On the other hand, $\Delta_0$ is not contained in $\widetilde{\mathcal{GP}}_g$. In fact, the generic $g$-nodal rational curve satisfies the Green-Lazarsfeld property $N_{\lfloor \frac{g-3}{2}\rfloor}$, that is, a much stronger property than projective normality, see \cite{V2}. Denoting by $\overline{\mathcal{GP}}_g$ the closure of the Petri divisor $\mathcal{GP}_g$ inside $\mm_g$, we thus have an equality of effective divisors on $\mm_g$
$$\widetilde{\mathcal{GP}}_g=\overline{\mathcal{GP}}_g+\sum_{i=1}^{\lfloor \frac{g}{2}\rfloor} b_i \Delta_i,$$
where $b_i\geq 1$, for all $i\geq 1$. The class of $\widetilde{\mathcal{GP}}_g$ can now be easily determined.

\vskip 3pt

\noindent \emph{Proof of Theorem \ref{petriclass}.}
We apply Theorem \ref{classdiv1} in the case of the morphism $\phi:\mbox{Sym}^2(\mathbb E)\rightarrow \mathbb F$ over $\mm_g$ given by multiplication. We have $c_1(\mathbb E)=\lambda$, whereas by the Grothendieck-Riemann-Roch calculation carried out in \cite{Mum} Theorem 5.10, one has
$c_1(\mathbb F)=\lambda+\kappa_1=13\lambda-\delta$.
\hfill $\Box$

\vskip 4pt

The Petri divisor decomposes into components depending on the degree of the pencil for which the Petri Theorem fails. For
$\lfloor \frac{g+2}{2}\rfloor \leq k\leq g-1$, we denote by $D_{g,k}$ the locus of curves $[C]\in \cM_g$ for which there exists a \emph{base point free} pencil $A\in W^1_k(C)$
such that $\mu(A)$ is not injective. It is shown in \cite{F3} that $D_{g,k}$ has at least one divisorial component. In light of Proposition \ref{gp1}, we have the decomposition
\begin{equation}\label{petridec}
\overline{\mathcal{GP}}_g=\sum_{k=\lfloor \frac{g+2}{2}\rfloor}^{g-1} a_{g,k} \overline{D}_{g,k}.
\end{equation}

It is an interesting open question to determine the classes $[\overline{D}_{g,k}]\in CH^1(\mm_g)$ and their multiplicities $a_{g,k}$. For birational geometry application, it is more relevant to compute the slopes $s(\overline{D}_{g,k})$.  Few of the individual divisors $D_{g,k}$ are well understood.

\vskip 4pt

By the proof of Proposition \ref{gp1}, the divisor $D_{g,g-1}$ consists of curves with an even theta-characteristic $\vartheta\in \mbox{Pic}^{g-1}(C)$ such that
$h^0(C,\vartheta)\geq 2$.  The class of its compactification in $\mm_g$ has been computed in \cite{T} and we have:
\begin{equation}\label{teixi}
[\overline{D}_{g,g-1}]=2^{g-3}\Bigl((2^g+1)\lambda-2^{g-3}\delta_0-\sum_{i=1}^{\lfloor \frac{g}{2}\rfloor} (2^{g-i}-1)(2^i-1)\delta_i\Bigr).
\end{equation}

When $k$ is minimal, for odd   $g=2k-1$, the locus $D_{g,k}$ is the Hurwitz divisor of curves of gonality at most $k$. Its compactification is the image of the space of admissible covers $\hh_k$ defined in the Introduction. Harris and Mumford \cite{HM} computed its class, on their way to show that $\mm_g$ is of general type for large odd genus $g\geq 25$:
\begin{equation}\label{harmum}
[\overline{D}_{2k-1,k}]=\frac{1}{(2k-2)(2k-3)}{2k-2\choose k-1}\Bigl(6(k+1)\lambda-k\delta_0-\sum_{i=1}^{k-1} 3i(2k-i-1)\delta_i\Bigr).
\end{equation}

For even genus $g=2k$, the divisor $\overline{D}_{2k,k+1}$ can be viewed as the branch map of the generically finite cover $\hh_{2k,k+1}\rightarrow \mm_{2k}$ from the space of admissible covers of degree $k+1$. The calculation of its class in \cite{EH} Theorem 2 has been instrumental in proving that $\mm_g$ is of general type for even genus $g\geq 24$:
\begin{equation}\label{branch}
[\overline{D}_{2k,k+1}]=\frac{2(2k-2)!}{(k-1)!(k+1)!}\Bigl((6k^2+13k+1)\lambda-k(k+1)\delta_0-(2k-1)(3k+1)\delta_1-\cdots \Bigr).
\end{equation}

\vskip 3pt

The only case when $k$ is not extremal has been treated in \cite{F3} and it concerns the divisor $D_{2k-1,k+1}$. It is shown in \cite{F3} Corollary 0.6 that its slope equals
\begin{equation}\label{nextto}
s(\overline{D}_{2k-1,k+1})=\frac{6k^2+14k+3}{k(k+1)}.
\end{equation}

\vskip 5pt

In the range $g\leq 7$, these known cases exhaust all Gieseker-Petri divisors and we can compare Theorem \ref{petriclass} with the previously mentioned formulas (\ref{teixi}), (\ref{harmum}), (\ref{branch}). We denote by $\widetilde{D}_{g,k}$ the closure of $D_{g,k}$ in $\widetilde{\cM}_g$. In order to determine the slope of $\overline{D}_{g,k}$, it suffices to compute the class $[\widetilde{D}_{g,k}]\in CH^1(\widetilde{\cM}_g)$, for as in the case of $\overline{\mathcal{GP}}_g$, the $\delta_0$-coefficient is smaller in absolute value than the higher boundary coefficients in the expansion of $[\overline{D}_{g,k}]$ in terms of the generators of $CH^1(\mm_g)$.

\vskip 4pt

For $g=4$, there is only one component and  we obtain
$[\widetilde{\mathcal{GP}}_4]=[\widetilde{D}_{4,3}]=34\lambda-4\delta_0 \in CH^1(\widetilde{\cM}_4)$.
For $g=5$, we obtain $[\widetilde{\mathcal{GP}}_5]=[\widetilde{D}_{5,4}]+4[\widetilde{D}_{5,3}]=4(41\lambda-5\delta_0)$, whereas for $g=6$, we find
$$[\widetilde{\mathcal{GP}}_6]=[\widetilde{D}_{6,5}]+4[\widetilde{D}_{6,4}]=8(112\lambda-14\delta_0)\in CH^1(\widetilde{\cM}_6).$$
Finally, in the case $g=7$, there are three Petri divisors and we obtain
$$[\widetilde{\mathcal{GP}}_7]=[\widetilde{D}_{7,6}]+4[\widetilde{D}_{7,5}]+16[\widetilde{D}_{7,4}]=96(55\lambda-7\delta_0)\in CH^1(\widetilde{\cM}_7).$$

Based on this formulas for small genus, we make the following conjecture, though we admit that the evidence for it is rather moderate.

\begin{conjecture}\label{multpetri}
One has $a_{g,k}=4^{g-1-k}$ for $\frac{g+2}{2}\leq k\leq g-1$, that is, the following holds:
$$[\widetilde{\mathcal{GP}}_g]=\sum_{i=1}^{\lceil \frac{g-2}{2}\rceil} 4^{i-1}[\widetilde{D}_{g,g-i}]\in CH^1(\widetilde{\cM}_g).$$
\end{conjecture}


\section{Effective divisors of small slope on $\mm_g$}\label{smallslopesec}

We now present an infinite series of effective divisors on $\mm_g$ of slope less than $6+\frac{12}{g+1}$, which recall, is the slope of all the
Brill-Noether divisors. We fix  integers $r\geq 3$ and $s\geq 1$ and set
$$g:=rs+s \ \mbox{ and } \ d:=rs+r.$$
Observe that $\rho(g,r,d)=g-(r+1)(g-d+r)=0$, hence by general Brill-Noether Theory a general curve of genus g has a finite number of linear systems of type
$\mathfrak g^r_d$. Let $\cM_g^{\sharp}$ the open substack of $\cM_g$ classifying smooth genus $g$ curves $C$ such that $W_{d-1}^r(C)=\emptyset$,  $W_d^{r+1}(C)=\emptyset$
and furthermore $H^1(C,L^{\otimes 2})=0$, for every $L\in W^r_d(C)$.
Then $\mbox{codim}(\cM_g-\cM_g^{\sharp},\cM_g)\geq 2$. For codimension one calculation, one makes no difference between $\cM_g$ and $\cM_g^{\sharp}$. We denote by $\mathfrak G^r_{g,d}$ the stack
parametrizing pairs $[C, L]$, with $[C]\in \cM_g^{\sharp}$ and $L\in W^r_d(C)$ is a necessarily complete and base point free linear system.
Let $$\sigma:\mathfrak G^r_{g,d}\rightarrow \cM_g^{\sharp}$$ be the natural
projection. It is known from general Brill-Noether Theory that there exists a unique irreducible component
of $\mathfrak G^r_{g,d}$ which maps dominantly onto $\cM_g$.

\vskip 4pt

We pick a general point $[C,L]\in \mathfrak{G}^r_{g,d}$ of the dominating component. It follows from the Maximal Rank Conjecture proved in this case in \cite{F2} or \cite{LOTZ} Theorem 1.4, that the multiplication map
$$\phi_{C,L}:\mbox{Sym}^2 H^0(C,L)\rightarrow H^0(C,L^{\otimes 2})$$
is surjective. Since $H^1(C,L^{\otimes 2})=0$,  by Riemann-Roch, the dimension of its kernel $I_{C,L}(2)$ equals
$$\mbox{dim } I_{C,L}(2)={r+2\choose 2}-(2d+1-g).$$
We impose the condition  that this number equal the codimension of the space $$\Sigma_{r+1}^{r-a-1}\subseteq \mbox{Sym}^2  H^0(C,L)$$ of quadrics
of rank at most $a+2$ (that is, corank $r-a-1$). Since $\mbox{codim}(\Sigma_{r+1}^{r-a-1})={r-a\choose 2}$, we obtain the following numerical constraint on $s$ and $r$:
\begin{equation}\label{numconda}
s=\frac{a(2r-1-a)}{2(r-1)}.
\end{equation}

For each $r$ and $s$ such that the equation (\ref{numconda}) is satisfied, we consider the locus
$$Z_{r,s}:=\Bigl\{[C,L]\in \mathfrak G^r_{g,d}: \exists \ 0\neq q\in I_{C,L}(2), \ \ \mbox{rk}(q)\leq a+2\Bigr\}$$
and set $D_{r,s}:=\sigma_*(Z_{r,s})$. Then $D_{r,s}$ is expected to be a divisor on $\cM_g$, that is, either it is a divisor in which case there exists a smooth curve $[C]\in \cM_g$ such that $I_{C,L}(2)$ contains no quadrics of rank at most $a+2$ for every $L\in W^r_d(C)$, or else $D_{r,s}=\cM_g$.   We shall determine the slope of the virtual class of its closure
in $\mm_g$.

\vskip 4pt

Before moving further, we discuss some solutions to equation (\ref{numconda}). If $a=r-1$ (that is, when one considers quadrics of maximal rank), then $r=2s$ and $g=s(2s+1)$. In this case $D_{2s,s}$ is
the locus of curves $[C]\in \cM_{s(2s+1)}$ for which there exists a linear series $L\in W^{2s}_{2s(s+1)}(C)$ such that the multiplication map $\phi_{C,L}:\mbox{Sym}^2 H^0(C,L)\rightarrow H^0(C,L^{\otimes 2})$ is not an isomorphism. This series of divisors has been studied in detail in \cite{F2} Theorem 1.5, as well as in \cite{Kh} and shown to contradict the Harris-Morrison Slope Conjecture \cite{HMo}.

\vskip 5pt

The first series of genuinely new examples  is when for an integer $\ell\geq 1$, we take
\begin{equation}\label{ser1}
s=4\ell-1, \ \ r=9\ell-2, \ \ a=2(3\ell-1), \ \mbox{ and } \ g=(4\ell-1)(9\ell-1).
\end{equation}

\vskip 3pt

Specializing to the case $\ell=1$, we obtain the following effective (virtual) divisor on $\cM_{24}$:
$$D_{7,3}:=\Bigl\{[C]\in \cM_{24}: \exists \ L\in W^7_{28}(C), \ \exists \ 0\neq q\in I_{C,L}(2), \ \mbox{ rk}(q)\leq 6\Bigr\}.$$

\vskip 4pt

A second series of examples is when for an integer $\ell\geq 1$, we take the following values
\begin{equation}\label{ser2}
s=3\ell+1, \ r=8\ell+3, \ a=4\ell+1, \ \mbox{ and } \ g=4(3\ell+1)(2\ell+1).
\end{equation}

The first example in this series appears produces an effective (virtual) divisor on $\cM_{48}$:
$$D_{11,4}:=\Bigl\{[C]\in \cM_{48}: \exists\  L\in W^{11}_{55}(C), \ \exists \ 0\neq q\in I_{C,L}(2), \ \mbox{ rk}(q)\leq 7\Bigr\}.$$

\vskip 4pt

We now describe the (virtual) divisor structure of $D_{r,s}$ and set up some notation that will help compute the class of their closure in $\mm_g$.
We introduce the partial compactification  $\widetilde{\cM}_g^{\sharp}$ defined as the union of $\cM_g^{\sharp}$ and the open substack $\Delta_0^{\sharp}\subseteq \Delta_0$ classifying $1$-nodal irreducible genus $g$ curves
$C'=C/p\sim q$, where $[C, p,q]\in \cM_{g-1,2}$ is a Brill-Noether
general $2$-pointed curve in the sense of \cite{EH} Theorem 1.1, together with all their
degenerations consisting of unions of a smooth genus $g-1$ curve and
a nodal rational curve. Note that $\widetilde{\cM}_g$ and $\widetilde{\cM}_g^{\sharp}$ agree outside a set of codimension $2$ and we
identify the Picard groups of the two stacks. We denote by $\widetilde{\mathfrak{G}}^r_{g,d}$ the parameter space of pairs $[C,L]$, where $[C]\in \cM_g^{\sharp}$ and
$L$ is a torsion free sheaf of rank $1$ and degree $d$ on $C$ such that $h^0(C,L)\geq r+1$. We still denote by $\sigma:\widetilde{\mathfrak{G}}_{g,d}^r\rightarrow \widetilde{\cM}_g^{\sharp}$ the proper forgetful morphism.

\vskip 4pt

We now consider the universal curve $\pi:\widetilde{\cM}_{g, 1}^{\sharp}\rightarrow \widetilde{\cM}_g^{\sharp}$
and denote by $\L$ a universal bundle on the
fibre product $\widetilde{\cM}_{g, 1}^{\sharp}\times_{\widetilde{\cM}_g^{\sharp}} \widetilde{\mathfrak G}^r_{g,d}$.   If
$$p_1:\widetilde{\cM}_{g, 1}^{\sharp}\times_{\widetilde{\cM}_g^{\sharp}} \widetilde{\mathfrak G}^r_{g,d} \rightarrow \widetilde{\cM}_{g,
1}^{\sharp} \ \ \mbox{ and }\  \ p_2:\widetilde{\cM}_{g,1}^{\sharp}\times_{\widetilde{\cM}_g^{\sharp}}  \widetilde{\mathfrak G}^r_{g,d}\rightarrow
\widetilde{\mathfrak G}^r_{g,d}$$ are the natural projections, then
$\cE:=p_{2 *}( \L)$ and $\cF:=p_{2 *}(\L^{\otimes 2})$  are locally free sheaves of ranks $r+1$ and $2d+1-g$ respectively.
Finally, we denote by
$$\phi:\mbox{Sym}^2(\cE)\rightarrow \cF$$
the sheaf morphism given by multiplication of sections.

\vskip 4pt

\begin{theorem}\label{pelda1}
Set $r=9\ell-2$ and $s=4\ell-1$, therefore $g=(4\ell-1)(9\ell-1)$, where $\ell\geq 1$. Then the virtual class of the closure of the divisor $D_{9\ell-2, 4\ell-1}$ inside $\mm_{(4\ell-1)(9\ell-1)}$ equals
$$s(\overline{D}_{9\ell-2, 4\ell-1})=\frac{a}{b},$$ where
$$a:=15116544\ell^8-30233088\ell^7+26605584\ell^6-13594392\ell^5+4419720\ell^4-899433\ell^3+$$
$$105656\ell^2-6101\ell+122$$
and
$$b:=2(9\ell-2)(9\ell-1)(15552\ell^6-25920\ell^5+17484\ell^4-6102\ell^3+1181\ell^2-107\ell+2).$$
In particular, $s(\overline{D}_{9\ell-2, 4\ell-1})<6+\frac{12}{g+1}$.
\end{theorem}

If we look at the difference between the slope of $\overline{D}_{9\ell-2,4\ell-1}$ and that of the Brill-Noether divisors we get a slightly simpler formula:
$$s(\overline{D}_{9\ell-2,4\ell-1})=6+\frac{12}{g+1}-$$
$$-\frac{(13\ell-2)(36\ell-13)(27\ell^2-19\ell+2)(36\ell^2-13\ell-1)}{2(9\ell-2)(9\ell-1)(15552\ell^6-25920\ell^5+17484\ell^4-6102\ell^3+
1181\ell^2-107\ell+2)(36\ell^2-13\ell+2)}.$$

\vskip 5pt

We now record the slope of the effective divisors in the second series of examples:

\vskip 4pt

\begin{theorem}\label{pelda2}
Set $r=8\ell+3$ and  $s=3\ell+1$, therefore $g=4(3\ell+1)(2\ell+1)$. Then the virtual class of the closure of the divisor $D_{8\ell+3, 3\ell+1}$ inside $\mm_{4(3\ell+1)(2\ell+1)}$ equals
$$s(\overline{D}_{8\ell+3, 3\ell+1})=6+\frac{12}{g+1}-$$
$$-\frac{(11\ell+5)(2\ell-1)(12\ell^2+10\ell+1)(24\ell^2+20\ell+3)}{(3\ell+2)(8\ell+3)(2304\ell^6+4128\ell^5+2992\ell^4+1128\ell^3+248\ell^2+41\ell+5)(24\ell^2+20\ell+5)}.$$
\end{theorem}

\noindent \emph{Proof of Theorems \ref{pelda1} and \ref{pelda2}.} We choose integers $r\geq 3$, $s,a\geq 1$ such that (\ref{numconda}) holds.
Recall that $d=rs+r$ and $g=rs+s$. We shall apply the techniques developed in \cite{F2} and \cite{Kh} in the context of Theorem \ref{divclass}. Recall that we have defined the vector bundle
morphism $\phi:\mbox{Sym}^2(\cE)\rightarrow \cF$ over the parameter space $\widetilde{\mathfrak{G}}_{g,d}^r$. Applying Theorem \ref{classdiv1}, if $Z_{r,s}$ is a divisor
on $\mathfrak{G}_{g,d}^r$,  then the class of its closure $\widetilde{Z}_{r,s}$ inside $\widetilde{\mathfrak{G}}_{g,d}^r$ is given by the formula
\begin{equation}\label{virtdiv1}
[\widetilde{Z}_{g,d}^r]=\alpha \Bigl(c_1(\cF)-\frac{2(2d+1-g)}{r+1} c_1(\cE)\Bigr).
\end{equation}

\vskip 3pt
We call the right hand side of the formula (\ref{virtdiv1}) the virtual class $[\widetilde{Z}_{g,d}^r]^{\mathrm{virt}}$ of the virtual divisor $\widetilde{Z}_{g,d}^r$.
Following \cite{F2} we introduce the following tautological divisor classes on $\widetilde{\mathfrak{G}}_{g,d}^r$:
$$\mathfrak{a}:=(p_2)_*\Bigl(c_1^2(\L)\Bigr), \ \ \mathfrak{b}:=(p_2)_*\Bigl(c_1(\L)\cdot c_1(\omega_{p_2})\Bigr) \ \mbox{ and }  \mathfrak{c}:=(p_2)_*\Bigl(c_1^2(\omega_{p_2})\Bigr)=\sigma^*(\kappa_1),$$
where we recall that $\kappa_1=12\lambda-\delta\in CH^1(\mm_g)$ is Mumford's class, see also \cite{Mum}.

\vskip 4pt

Since $R^1(p_2)_*(\L^{\otimes 2})=0$, applying Grothendieck-Riemann-Roch to  $p_2$, we compute
$$c_1(\cF)=\sigma^*(\lambda)-\mathfrak{b}+2\mathfrak{a}.$$
The push-forwards of the tautological classes $\mathfrak{a}, \mathfrak{b}$ and  $c_1(E)$ under the generically finite proper morphism $\sigma:\widetilde{\mathfrak{G}}_{g,d}^r\rightarrow \widetilde{\cM}_g^{\sharp}$ are determined in \cite{F2} Section 2 and \cite{Kh} Theorem 2.11 and we summarize the results: There exists an explicit constant $\beta\in \mathbb Z_{>0}$ such that
$$\sigma_*(\mathfrak{a})=\beta \frac{d}{(g-1)(g-2)}\Bigl((dg^2-2g^2+8d-8g+4)\lambda-(dg-2g^2+4d-3g+2)\delta_0\Bigr),$$
$$\sigma_*(\mathfrak{b})=\beta \frac{d}{g-1}\Bigl(6\lambda-\frac{\delta_0}{2}\Bigr)$$
and
$$\sigma_*(c_1(\cE))=\beta\Bigl(-\frac{r(r+2)(r^2 s^3+2rs^3-r^2s+6rs^2+s^3-2rs+6s^2-8r+3s-8)}{2(r+s+1)(rs+s-2)(rs+s-1)}\lambda+$$
$$+\frac{r(s-1)(s+1)(r+2)(r+1)(rs+s+4)}{12(r+s+1)(rs+s-2)(rs+s-1)}\delta_0 \Bigr).$$
We substitute these formulas in (\ref{virtdiv1}) and we obtain a closed formula for $[\widetilde{Z}_{r,s}]$. Substituting the particular values in Theorems \ref{pelda1} and
\ref{pelda2}, we obtain the claimed formulas for the slopes.

\hfill $\Box$

\vskip 3pt

We expect the virtual divisors constructed in Theorems \ref{pelda1} and \ref{pelda2} to be actual divisors for all $\ell$. We can directly confirm this expectation for all bounded $\ell$. We illustrate this in the case $\ell=1$.

\begin{theorem}\label{transv24}
The locus $D_{7,3}$ is a divisor on $\cM_{24}$, that is, for a general curve $C$ of genus $24$, the image curve $\varphi_L:C\hookrightarrow \PP^7$ lies on no quadric of rank at most $6$, for any linear system $L\in W^7_{28}(C)$.
\end{theorem}
\begin{proof}
By residuation, we have a birational isomorphism $\mathfrak{G}^7_{24,28}\cong \mathfrak{G}^2_{24,18}$ of parameter spaces over $\cM_{24}$. The latter space is a quotient of the Severi variety of plane curves of genus $24$ and degree $18$ which is known to be irreducible \cite{H}, hence $\mathfrak{G}^7_{24,28}$ is an irreducible, generically finite cover of $\cM_{24}$. To show that $D_{7,3}$ is a divisor, that is, $D_{7,3}\neq \cM_{24}$, it suffices to produce \emph{one} smooth curve $[C]\in \cM_{24}$ and \emph{one} very ample linear system
$L\in W^7_{28}(C)$ such that the image curve $\varphi_L:C\hookrightarrow \PP^7$ does not lie on any quadric of rank at most $6$. The curve we construct lies on  a rational surface $X$ in $\PP^7$ and has the property that all the quadrics containing $C$ also contain $X$.

\vskip 4pt

Precisely, we start with $16$ general points $p_1, \ldots, p_{16}\in \PP^2$. We embed the surface $X:=\mbox{Bl}_{16}(\PP^2)$ obtained by blowing-up these points in the space $\PP^7$ via the linear system
$$H=9h-3E_{1}-2\sum_{i=2}^{14} E_i-E_{15}-E_{16}\in \mbox{Pic}(X),$$
where $h$ is the hyperplane class and $E_i$ is the exceptional divisor corresponding to the point $p_i$, for $i=1,\ldots,16$. By direct computation we find
$$h^0(X,\OO_X(2))=h^0\Bigl(X,\OO_X\bigl(18h-6E_1-4\sum_{i=2}^{14}E_i-2E_{15}-2E_{16}\bigr)\Bigr)={20\choose 2}-{7\choose 2}-13{5\choose 2}-2{3\choose 2}=33.$$
By using \emph{Macaulay}, we check that $|H|$ embeds $X$ into $\PP^7$ and the map $\mbox{Sym}^2 H^0(\OO_X(1))\rightarrow H^0(\OO_X(2))$ is surjective, hence
$\mbox{dim } I_{X,\OO_X(1)}(2)=3$ and $H^1(\PP^7, \mathcal{I}_{X/\PP^7}(2))=0$.  We check furthermore with \emph{Macaulay} that $I_{X,\OO_X(1)}(2)\cap \Sigma_8^2=\emptyset$, that is, $X\subseteq \PP^7$ lies on
no quadric of rank at most $6$.

\vskip 4pt

We construct a curve $C\subseteq X$ as a general element of the linear system
$$C\in \Bigl|20h-6E_1-5\sum_{i=2}^{13} E_i-4E_{14}-3E_{15}-3E_{16}\Bigr|.$$
Then $C\cdot H=28$ and we check by \emph{Macaulay} that such a curve $C$ is smooth. In particular,
it follows that $g(C)=1+\frac{1}{2} C\cdot (C+K_X)=24$. Furthermore, one has an exact sequence
$$0\longrightarrow I_{X,\OO_X(1)}(2)\longrightarrow I_{C,\OO_C(1)}(2)\longrightarrow H^0(X,\OO_X(2H-C))\longrightarrow 0,$$
obtained by taking cohomology in the exact sequence $0\rightarrow \mathcal{I}_{X/\PP^{7}}(2)\rightarrow \mathcal{I}_{C/\PP^7}(2)\rightarrow \OO_X(2H-C)\rightarrow 0$, where we use once more that $H^1(\PP^7, \mathcal{I}_{X/\PP^7}(2))=0$.
Since $H^0(X,\OO_X(2H-C))=0$, this induces an isomorphism $I_{X,\OO_X(1)}(2)\cong I_{C,\OO_C(1)}(2)$. This shows that the smooth curve $C\subseteq \PP^7$ lies on no quadric of rank at most $6$, which finishes the proof{\footnote{The \emph{Macaulay} file containing all the computations appearing in this proof can be found online at https://www.mathematik.hu-berlin.de/\~ \ farkas/computations-gen24.m2.}}.

\end{proof}

\section{The slope of $\mm_{12}$}

We explain in this section how using Theorems \ref{degpencint} and \ref{prop:disc2} one can construct an effective divisor on $\mm_{12}$ having slope less than $6+\frac{12}{g+1}$.

\vskip 3pt

 A general curve $[C]\in \cM_{12}$ has finitely many linear systems
$L\in W^5_{15}(C)$. As already pointed out, the multiplication map
$\phi_{C,L}: \mbox{Sym}^2 H^0(C, L)\rightarrow H^0(C, L^{\otimes 2})$ is surjective for each $L\in W^5_{15}(C)$, in particular $\PP_{L}:=\PP\bigl(I_{C,L}(2)\bigr)$ is a pencil of quadrics
in $\PP^5$ containing the curve $\varphi_L:C\hookrightarrow \PP^5$. By imposing the condition that the pencil $\PP_L$ be degenerate, we produce a divisor on $\mm_{12}$, whose class we ultimately compute.

\vskip 3pt

\noindent \emph{Proof of Theorem \ref{sec12}.} We retain the notation of the previous section and recall that $\sigma:\widetilde{\mathfrak{G}}_{12,15}^5\rightarrow \widetilde{\cM}_{12}^{\sharp}$ denotes the proper forgetful morphism from the parameter space of generalized linear series $\mathfrak g^5_{15}$ onto (an open subset of) the moduli space of irreducible curves of genus $12$. Furthermore, we retain the same notation for the tautological bundles $\cE$ and $\cF$ over $\widetilde{\mathfrak{G}}^5_{12,15}$, as well as for the vector bundle morphism $\phi:\mbox{Sym}^2(\cE)\rightarrow \cF$, globalizing the multiplication maps $\phi_{C,L}$, as $[C, L]$ varies over $\widetilde{\mathfrak{G}}_{12,15}^5$. In particular $\PP_L \cong \PP\bigl(\mbox{Ker}(\phi_{C,L})\bigr)$, for every $[C,L]\in \widetilde{\mathfrak{G}}_{12,15}^5$. Noting that $\mbox{rk}(\cE)=6$ and $\mbox{rk}(\cF)=19$, we apply Proposition \ref{prop:disc2}.  The virtual class  of the locus $Z$ of pairs $[C,L]\in \widetilde{\mathfrak{G}}_{12,15}^5$ such that $\PP_L$ is a degenerate pencil equals
$$[Z]^{\mathrm{virt}}=10\Bigl(6c_1(\cF)-38c_1(\cE)\Bigr)\in CH^1(\widetilde{\mathfrak{G}}_{12,15}^5).$$
The pushforward classes $\sigma_*(c_1(\cE))$ and $\sigma_*(c_1(\cF))$ have been described in the proof of Theorems \ref{pelda1} and \ref{pelda2}. After easy manipulations, we compute the  class $[\overline{\mathfrak{Dp}}_{12}]^{\mathrm{virt}}:=\sigma_*([Z]^{\mathrm{virt}})\in CH^1(\widetilde{\cM}_{12}^{\sharp})$.

\vskip 4pt

It remains to establish that $Z$ is indeed a divisor inside $\widetilde{\mathfrak{G}}^{5}_{12,15}$. To that end, we observe that one has a birational isomorphism
$\mathfrak{G}^5_{12,15}\cong \mathfrak{G}^1_{12,7}$. The latter being the Hurwitz space of degree $7$ covers of $\PP^1$, it is well-known to be irreducible, hence $\mathfrak{G}^5_{12,15}$ is irreducible as well. Therefore it suffices to exhibit one projectively normal smooth curve $C\subseteq \PP^5$ of genus $12$ and degree $15$, such that $\PP_{\OO_C(1)}$ is non-degenerate.  This is achieved in a way similar to the proof of Theorem \ref{transv24}, by choosing $C$ to lie on a particular rational surface.

\vskip 3pt

We pick $11$ general points $p_1, \ldots, p_{11}\in \PP^2$. We embed the surface $X:=\mbox{Bl}_{11}(\PP^2)$ obtained by blowing-up these points in $\PP^5$ via the linear system
$$H=5h-2E_1-2E_2-\sum_{i=3}^{11} E_i\in \mbox{Pic}(X),$$
where $h$ is the hyperplane class and $E_i$ is the exceptional divisor corresponding to the point $p_i$, for $i=1,\ldots,11$. We compute
$h^0(X,\OO_X(2))=19$ and $\mbox{dim } I_{X,\OO_X(1)}(2)=2$. We check furthermore with \emph{Macaulay} that the pencil $\PP_{\OO_X(1)}$ is non-degenerate.

\vskip 3pt

We construct a curve $C\subseteq X$ as a general element of the following linear system on $X$
$$C\in \Bigl|10h-4E_1-4E_2-3E_3-3E_4-2\sum_{i=5}^{10} E_i-E_{11}\Bigr|.$$
Then $C$ is a smooth curve of genus $12$ with $C\cdot H=15$. Since   $H^0(X,\OO_X(2H-C))=0$,
we have an isomorphism $I_{X,\OO_X(1)}(2)\cong I_{C,\OO_C(1)}(2)$, showing that the pencil $\PP_{\OO_C(1)}$ is non-degenerate.

\hfill $\Box$

\section{Tautological classes on the moduli space of polarized $K3$ surfaces}\label{sectk3}

For a positive integer $g$, we denote by $\cF_g$ the moduli space of quasi-polarized $K3$ surfaces of genus $g$ classifying pairs $[X,L]$, where $X$ is a smooth $K3$ surface
and $L\in \mbox{Pic}(S)$ is a big and nef line bundle with $L^2=2g-2$. Via the Torelli Theorem for $K3$ surfaces, one can realize $\cF_g$ as the quotient $\Omega_g/\Gamma_g$ of a $19$-dimensional symmetric domain $\Omega_g$ by an arithmetic subgroup $\Gamma_g$ of $SO(3,19)$.

\vskip 3pt

We denote by $\pi:\cX\rightarrow \cF_g$ the universal polarized $K3$ surface of genus $g$ and by $\L\in \mbox{Pic}(\cX)$ a universal polarization line bundle. Note that
$\L$ is not unique, for it can be twisted by the pull-back of any line bundle coming from $\cF_g$. Recall that the Hodge bundle on $\cF_g$ is defined by
$$\lambda:=\pi_*(\omega_{\pi})\in \mbox{Pic}(\cF_g).$$
Following \cite{MOP}, for non-negative integers $a,b$ we also consider the $\kappa$ classes on $\cF_g$, by setting
$$\kappa_{a,b}:=\pi_*\Bigl(c_1(\L)^{a}\cdot c_2(\mathcal{T}_{\pi})^b\Bigr)\in CH^{a+2b-2}(\cF_g).$$
We shall concentrate on the codimension $1$ tautological classes, that is, on
$\kappa_{3,0}$ and $\kappa_{1,1}$.
Replacing $\L$ by $\widetilde{\L}:=\L\otimes \pi^*(\alpha)$, where $\alpha\in \mbox{Pic}(\cF_g)$, the classes $\kappa_{3,0}$ and $\kappa_{1,1}$ change as follows:
$$\widetilde{\kappa}_{3,0}=\kappa_{3,0}+6(g-1)\alpha \ \mbox{ and } \ \widetilde{\kappa}_{1,1}=\kappa_{1,1}+24\alpha.$$
It follows that the following linear combination of $\kappa$ classes
$$\gamma:=\kappa_{3,0}-\frac{g-1}{4}\kappa_{1,1}\in CH^1(\cF_g)$$
is well-defined and independent of the choice of a Poincar\'e bundle on $\cX$.

\vskip 3pt

\subsection{$K3$ surfaces and rank $4$ quadrics}\label{k34sub}

Recall that in the Introduction we have introduced the Noether-Lefschetz divisors $D_{h,d}$ consisting of quasi-polarized $K3$ surfaces $[X,L]\in \cF_g$ such that there exists a primitive embedding of a rank $2$ lattice
$\mathbb Z\cdot L\oplus \mathbb Z\cdot D\subseteq \mathrm{Pic}(X)$, where $D\in \mbox{Pic}(X)$ is a class with $D\cdot L=d$ and $D^2=2h-2$. In what follows, we fix a quasi-polarized $K3$ surface $[X,L]\in \cF_g$ and consider the map
$$\varphi_L:X\rightarrow \PP^g$$
induced by the polarization.
We recall a few classical results on linear systems on $K3$ surfaces. Since $L$ is big and nef, using \cite{SD} Proposition 2.6, we find that $L$ is base point free unless there exists an elliptic curve $E\subseteq X$ with $E\cdot L=1$. In this case, $L=gE+\Gamma$, where $\Gamma^2=-2$ and $E\cdot \Gamma=1$. This case corresponds to the NL divisor $D_{1,1}$. If $L$ is base point free, then $L$ is not very ample if and only if there is a divisor $E\in \mbox{Pic}(X)$ with $E^2=-2$ and $E\cdot L=0$ (which corresponds to the NL divisor $D_{0,0}$), or there is a divisor $E\in \mbox{Pic}(X)$ with $E^2=0$ and $E\cdot L=2$, which corresponds to the NL divisor $D_{1,2}$.

\vskip 3pt

When $[X,L]\in D_{0,0}$, the morphism $\varphi_L$ contracts the smooth rational curve $\Gamma$. The NL divisor $D_{1,2}$ consists of \emph{hyperelliptic} $K3$ surfaces, for in this case $\varphi_L$ maps  $X$ with degree $2$ onto a surface of degree $g-1$ in $\PP^g$. Furthermore, for $[X,L]\in \cF_g-\bigl(D_{0,0}\cup D_{1,1}\cup D_{1,2}\bigr)$, it is shown in \cite{SD} Theorem 6.1 that the multiplication map
$$\phi_{X,L}:\mbox{Sym}^2 H^0(X,L)\rightarrow H^0(X,L^{\otimes 2})$$
is surjective. By Riemann-Roch, $h^0(X,L^{\otimes 2})=\chi(X,\OO_X)+2L^2=4g-2$ and we obtain
$$\mbox{dim } I_{X,L}(2)={g+2\choose 2}-(4g-2)={g-2\choose 2}=\mbox{codim}(\Sigma_{g+1}^{g-3}).$$

Recall that we have defined in the Introduction the locus $D_g^{\mathrm{rk} 4}$ of quasi-polarized $K3$ surfaces $[X,L]\in \cF_g$ such that the image $\varphi_L(X)\subseteq \PP^g$  lies on a rank $4$ quadric.

\begin{proposition}\label{nldiv3}
The locus $D_g^{\mathrm{rk} 4}$ is a Noether-Lefschetz divisor on $\cF_g$. Set-theoretically, it consists of the quasi-polarized $K3$ surfaces $[X,L]\in \cF_g$, for which there exists a decomposition $L=D_1+D_2$ in $\mathrm{Pic}(X)$, with $h^0(X,D_i)\geq 2$, for $i=1,2$.
\end{proposition}
\begin{proof}
Suppose the embedded $K3$ surface $X\hookrightarrow \PP^g$ lies on a quadric $Q\subseteq \PP^g$ of rank at most $4$. Assume $\mbox{rk}(Q)=4$, hence $\mbox{Sing}(Q)\cong \PP^{g-4}$. Then $Q$ is isomorphic to the inverse image of $\PP^1\times \PP^1$ under the projection $p_{\mathrm{Sing}(Q)}:\PP^{g}\dashrightarrow \PP^3$ with center $\mbox{Sing}(Q)$. Accordingly, $Q$ has two rulings which cut out line bundles $D_1$ and $D_2$ on $X$ such that $h^0(X,D_i)\geq 2$ and $L=D_1+D_2$. The argument
is clearly reversible.
\end{proof}

For $n\geq 1$, we introduce the following tautological bundles $$\cU_n:=\pi_*(\L^{\otimes n})$$
on $\cF_g$. Note that $R^i \pi_*(\L^{\otimes n})=0$ for $i=1,2$, hence $\cU_n$ is locally free and
$\mbox{rk}(\cU_n)=2+n^2(g-1)$.

\begin{proposition}\label{chernn1}
The following formula holds for every $n\geq 1$:
$$c_1(\cU_n)=\frac{n}{12}\kappa_{1,1}+\frac{n^3}{6}\kappa_{3,0}-\Bigl(\frac{n^2}{2}(g-1)+1\Bigr)\lambda\in CH^1(\cF_g).$$
\end{proposition}
\begin{proof} We apply Grothendieck-Riemann-Roch to the universal $K3$ surface $\pi:\cX\rightarrow \cF_g$ and write:
$$\mbox{ch}\bigl(\pi_{!}\L^{\otimes n}\bigr)=\pi_*\Bigr[\Bigl(1+nc_1(\L)+\frac{n^2}{2}c_1^2(\L)+\frac{n^3}{6}c_1^3(\L)+\cdots\Bigr)\cdot $$$$\Bigl(1-\frac{1}{2}c_1(\omega_{\pi})+\frac{1}{12}\bigl(c_1^2(\omega_{\pi})+c_2(\Omega_{\pi})\bigr)-
\frac{1}{24} c_1(\omega_{\pi})c_2(\Omega_{\pi})+\cdots \Bigr)\Bigr].$$
Note that $\kappa_{2,0}=\pi_*(c_1^2(\L))=2g-2\in CH^0(\cF_g)$, hence by looking at degree $2$ terms in this formula, we find $\kappa_{0,1}=24$.
We now consider degree $3$ terms that get pushed forward under $\pi$, and use that $c_1(\Omega_{\pi})=\pi^*(\lambda)$, hence $\pi_*\bigl(c_1(\L)\cdot c_1^2(\omega_{\pi})\bigr)=0$.
Collecting terms, we obtained the desired formula.
\end{proof}

\vskip 3pt

We are now in a position to compute the class of the Noether-Lefschetz  divisor $D_g^{\mathrm{rk} 4}$.

\vskip 2pt

\noindent \emph{Proof of Theorem \ref{rank4intro}.}
On the moduli space $\cF_g$ we consider the vector bundle morphism $$\phi:\mbox{Sym}^2 (\cU_1)\rightarrow \cU_2.$$ The divisor $D_g^{\mathrm{rk} 4}$ coincides with the locus where the kernel of $\phi$ contains a rank $4$ quadric. Applying Theorem \ref{divclass}, we find the formula
$$[D_g^{\mathrm{rk} 4}]=A_{g+1}^{g-3}\Bigl(c_1(\cU_2)-\frac{8g-4}{g+1}c_1(\cU_1)\Bigr).$$
In view of Proposition \ref{chernn1},
$c_1(\cU_1)=\frac{1}{12}\kappa_{1,1}+\frac{1}{6}\kappa_{3,0}-\frac{g+1}{2}\lambda$ and $c_1(\cU_2)=\frac{1}{6}\kappa_{1,1}+\frac{4}{3}\kappa_{3,0}-(2g-1)\lambda$.
Substituting, we obtain the claimed formula.
\hfill $\Box$

\vskip 4pt

\subsection{Koszul cohomology of polarized $K3$ surfaces of odd genus}\label{subsectkosz}

\hfill
\vskip 3pt

Theorem \ref{rank4intro} shows that a certain linear combination of the classes $\lambda$ and $\gamma$ lies in the span of $NL$ divisors. To conclude that both $\lambda$ and $\gamma$ are of NL-type, we find another linear combination of these two classes, that is guaranteed to be supported on NL divisors.  To that end, for odd genus, we use Voisin's solution \cite{V1}, \cite{V2}  to the Generic Green's Conjecture on syzygies of canonical curves.

\vskip 3pt

We fix a quasi-polarized $K3$ surface $[X,L]\in \cF_g-D_{1,1}$, so that $L$ is globally generated and we consider the induced morphism $\varphi_L:X\rightarrow \PP^g$. We introduce the coordinate ring $$\Gamma_X(L):=\bigoplus_{n\geq 0} H^0(X,L^{\otimes n}),$$ viewed as a graded module over the polynomial algebra $S:=\mbox{Sym } H^0(X,L)$. In order to describe the minimal free resolution of $\Gamma_X(L)$,
for integers $p,q\geq 0$, we introduce the Koszul cohomology group $$K_{p,q}(X,L)=\mbox{Tor}_S^p(\Gamma_X(L),\mathbb C)_{p+q}$$ of $p$-th order syzygies of weight $q$ of the pair $[X,L]$. We set $b_{p,q}(X,L):=\mbox{dim } K_{p,q}(X,L)$. For an introduction to Koszul cohomology in algebraic geometry, we refer to \cite{G} and \cite{AN}.

\vskip 4pt

The graded minimal free $S$-resolution of $\Gamma_X(L)$ has the following shape:
$$ 0\longleftarrow \Gamma_X(L)\longleftarrow F_{0}\longleftarrow F_1\longleftarrow \cdots \longleftarrow F_{g-3}\longleftarrow F_{g-2}\longleftarrow 0,$$
where $F_p=\bigoplus_{q>0} S(-p-q)\otimes K_{p,q}(X,L)$, for all $p\leq g-2$.

\vskip 4pt

The resolution is self-dual in the sense that $K_{p,q}(X,L)^{\vee} \cong K_{g-2-p,3-q}(X,L)$, see \cite{G} Theorem 2.c.6. This shows that the \emph{linear strand} of the Betti diagram of $[X,L]$ corresponding to the case $q=1$ is dual to the \emph{quadratic strand} corresponding to the case $q=2$.  In \cite{V2}, in her course of proving Green's Conjecture for general curves \cite{G}, Voisin determined completely the shape of the minimal resolution of a generic quasi-polarized $K3$ surface $[X,H]\in \cF_g$ of odd genus $g=2i+3$. We summarize in the following table  the relevant information contained
in the rows of linear and quadratic syzygies of the Betti table.
\begin{table}[htp!]
\begin{center}
\begin{tabular}{|c|c|c|c|c|c|c|c|c|}
\hline
$1$ & $2$ & $\ldots$ & $i-1$ & $i$ & $i+1$ & $i+2$  & $\ldots$ & $2i$\\
\hline
$b_{1,1}$ & $b_{2,1}$ & $\ldots$ & $b_{i-1,1}$ & $b_{i,1}$ & 0 & 0 &  $\ldots$ & 0 \\
\hline
$0$ &  $0$ & $\ldots$ & $0$ & $0$ & $b_{i+1,2}$ & $b_{i+2,2}$ & $\ldots$ & $b_{2i,2}$\\
\hline
\end{tabular}
\end{center}
    \caption{The Betti table of a general polarized $K3$ surface of genus $g=2i+3$}
\end{table}

\vskip 3pt

The crux of Voisin's proof is showing $K_{i+1,1}(X,L)=0$, which  implies $K_{p,1}(X,L)=0$ for $p>i$. Then by duality,
the second row of the resolution has the form displayed above.

\vskip 4pt

Our strategy is to treat this problem variationally and consider the locus of polarized $K3$ surfaces with extra syzygies, that is,
$$\mathfrak{Kosz}_g:=\Bigl\{[X,L]\in \cF_g:K_{i+1,1}(X,L)\neq 0\Bigr\}.$$
We shall informally refer to $\mathfrak{Kosz}_g$ as the \emph{Koszul divisor} on $\cF_g$, where $g=2i+3$. It is shown in \cite{AN} Corollary 2.17 that the group $K_{i+1,1}(X,L)$ of linear syzygies has the following interpretation
$$K_{i+1,1}(X,L)\cong K_{i,2}\bigl(I_{X,L}, H^0(X,L)\bigr),$$
where $I_{X,L}:=\oplus_{k} I_{X,L}(k)$ is the ideal of $X\subseteq \PP^g$, cf. Definition \ref{quadrk}, viewed as a graded $\mbox{Sym } H^0(X,L)$-module. Thus, one has the following identification
\begin{equation}\label{koszint}
K_{i+1,1}(X,L)\cong \mbox{Ker}\Bigl\{\bigwedge^{i} H^0(X,L)\otimes I_{X,L}(2)\rightarrow \bigwedge^{i-1} H^0(X,L)\otimes I_{X,L}(3)\Bigr\}\cong H^0\Bigl(\PP^g, \Omega^{i}_{\PP^{g}}(i+2)\otimes \mathcal{I}_{X/\PP^g}\Bigr),
\end{equation}
where the map in question is given by the Koszul differential. The last identification in (\ref{koszint}) is obtained by taking global sections in the exact sequence on $\PP^g$
$$0\longrightarrow \bigwedge^i M_{\PP^g}\otimes \I_{X/\PP^g}(2)\longrightarrow \bigwedge^i H^0(\PP^g, \OO_{\PP^{g}}(1))\otimes \I_{X/\PP^g}(2)\longrightarrow \bigwedge^{i-1} M_{\PP^g}\otimes \I_{X/\PP^g}(3)\longrightarrow 0, $$
where $M_{\PP^g}:=\Omega_{\PP^g}(1)$.
More generally, we introduce the \emph{Lazarsfeld bundle} of $[X,L]$ as the kernel of the evaluation map of global sections, that is,
\begin{equation}\label{lazb}
0\longrightarrow M_{L}\longrightarrow H^0(X,L) \otimes \mathcal{O}_X \longrightarrow L\longrightarrow 0.
\end{equation}
Note that $M_L=\Omega_{\PP^g|X}(1)$.  Via (\ref{koszint}), $[X,L]\in \mathfrak{Kosz}_g$ if and only the restriction map is not injective:
\begin{equation}\label{morph1}
H^0\Bigl(\PP^g, \bigwedge^i M_{\PP^g}(2)\Bigr)\rightarrow H^0\Bigl(X, \bigwedge^i M_L\otimes L^{2}\Bigr).
\end{equation}
The key observation  is that the two spaces appearing in (\ref{morph1}) have the same dimension, which leads to representing $\mathfrak{Kosz}_g$ as the degeneracy locus
of a morphism between two vector bundles of the \emph{same} rank over $\cF_g$.

\vskip 3pt

We collect a few technical results that will come up in the following calculations:

\begin{lemma}\label{calcul}
Let $[X,L]\in \cF_{2i+3}$ be a quasi-polarized $K3$ surface such that $L$ is base point free.
\vskip 3pt

\noindent (1) $H^1\bigl(X,\bigwedge ^j M_L\otimes L^{i+2-j}\bigr)=0$, for $j=0, \ldots, i$.

\vskip 3pt

\noindent (2) $h^0\bigl(X,\bigwedge^i M_L\otimes L^{2}\bigr)=h^0\bigl(\PP^{2i+3}, \bigwedge^i M_{\PP^{2i+3}}(2))=(i+1){2i+5\choose i+2}.$
\end{lemma}

\begin{proof}
It is proved in \cite{Ca} Corollary 1 that under our assumption, the vector bundle $M_L$ is $\mu_L$-semistable. This implies that  $\bigwedge^j M_L\otimes L^{2+i-j}$ is $\mu_L$-semistable for all $i$ and $j$ as well. We take cohomology in the exact sequence
$$0\longrightarrow \bigwedge^{j+1}M_L\otimes L^{i+1-j}\longrightarrow \bigwedge^{j+1} H^0(X,L)\otimes L^{i+1-j}\longrightarrow \bigwedge^j M_L\otimes L^{i+2-j}\longrightarrow 0.$$
Since $H^1(X,L^{i+1-j})=0$ and $H^2(X, L^{i+1-j})=0$ for $j\leq i$, we obtain the isomorphism
$$H^1\bigl(X,\bigwedge^j M_L\otimes L^{i+2-j})\cong H^2\bigl(X,\bigwedge^{j+1} M_L\otimes L^{i+1-j}\bigr).$$
Since $\mbox{rk}(M_L)=g$ and $c_1(M_L)=-L$, by standard Chern class calculation, we find
$$\mu_L\Bigl(\bigwedge^{j+1} M_L\otimes L^{i+1-j}\Bigr)=\frac{i+2}{2i+3}(2i-2j+1)>0,$$
which establishes $H^2\bigl(X,\bigwedge^{j+1} M_L\otimes L^{i+1-j}\bigr)=0$ by the stability of the vector bundle in question.

\vskip 3pt

The fact that $h^0\bigl(\PP^g, \bigwedge^i M_{\PP^{2i+3}}(2)\bigr)=h^0\bigl(\PP^g, \Omega_{\PP^{2i+3}}(i+2))=(i+1){2i+5\choose i+2}$ follows directly from Bott's formula on the cohomology of spaces of twisted holomorphic forms on projective spaces, see e.g. \cite{OSS} page 4. To compute the last quantity appearing, noting that $c_2(M_L)=2g-2$, after a Riemann-Roch calculation on $X$, we obtain
$$h^0\bigl(X,\bigwedge^i M_L\otimes L^2)=\chi\bigl(X,\bigwedge^i M_L\otimes L^2\bigr)=(i+1){2i+5\choose i+2},$$
where we have used the standard formulas $c_1\bigl(\bigwedge^i M_L\bigr)={2i+2\choose i-1}c_1(M_L)$ and
$$c_2\bigl(\bigwedge^i M_L\bigr)=\frac{1}{2}{2i+2\choose i-1}\Bigl({2i+2\choose i-1}-1\Bigr)c_1^2(M_L)+{2i+1\choose i-1} c_2(M_L).$$
\end{proof}

\vskip 3pt

Taking exterior powers in the short exact sequence (\ref{lazb}) and using the first part of Lemma \ref{calcul}, for $j=0,\ldots,i$, we obtain
the exact sequences, valid for $[X,L]\in \cF_g-D_{1,1}$:
$$0\longrightarrow H^0\bigl(X,\bigwedge^jM_L\otimes L^{i+2-j}\bigr)\longrightarrow \bigwedge^j H^0(X,L)\otimes H^0(X,L^{i+2-j})\longrightarrow H^0\bigl(X,\bigwedge^{j-1}M_L\otimes L^{i+3-j}\bigr)\longrightarrow 0.$$
Globalizing these exact sequences over the moduli space, for $j=0,\ldots, i$, we define inductively the vector bundles $\G_{j,i+2-j}$ over $\cF_g$ via the exact
sequences
\begin{equation}\label{exseqg}
0\longrightarrow \G_{j,i+2-j}\longrightarrow \bigwedge^j \cU_1\otimes \cU_{i+2-j}\longrightarrow \G_{j-1,i+3-j}\longrightarrow 0,
\end{equation}
starting from $\G_{0,i+2}:=\cU_{i+2}$.

\vskip 3pt

Similarly, taking exterior powers in the Euler sequence on $\PP^g$, we find the exact sequences
$$0\longrightarrow H^0\bigl(\bigwedge^j M_{\PP^g}(i+2-j)\bigr)\longrightarrow \bigwedge^j H^0(\OO_{\PP^g}(1))\otimes H^0(\OO_{\PP^g}(i+2-j))\longrightarrow
H^0\bigl(\bigwedge^{j-1} M_{\PP^g}(i+3-j)\bigr)\longrightarrow 0,$$
which can also be globalizes to exacts sequences over $\cF_g$. We define inductively the vector bundles $\H_{j,i+2-j}$ for $j=0, \ldots, i$, starting from
$\H_{0,i+2}:=\mbox{Sym}^{i+2}(\cU_1)$ and then via the exact sequences
\begin{equation}\label{exseqh}
0\longrightarrow \H_{j,i+2-j}\longrightarrow \bigwedge^j \cU_1\otimes \mbox{Sym}^{i+2-j}(\cU_1)\longrightarrow \H_{j-1,i+3-j}\longrightarrow 0.
\end{equation}
In particular, there exist restriction morphisms $\H_{j,i+2-j}\rightarrow \G_{j,i+2-j}$ for all $j=0, \ldots, i$. Setting $j=i$, we observe that the second part of Lemma
\ref{calcul} yields $\mbox{rk}(\H_{i,2})=\mbox{rk}(\G_{i,2})$, and the degeneracy locus of the morphism
$$\phi:\H_{i,2}\rightarrow \G_{i,2}$$
is precisely the locus $\mathfrak{Kosz}_g$ of quasi-polarized $K3$ surfaces having  extra syzygies.

\begin{proposition}\label{kosz1}
The locus $\mathfrak{Kosz}_g$ is an effective divisor on $\cF_g$ of NL type.
\end{proposition}
\begin{proof}
Let $[X,L]\in \mathcal{F}_g$ be a quasi-polarized $K3$ surface with $\mbox{Pic}(X)=\mathbb Z\cdot L$ and choose a general curve $C\in |L|$.  Using the Koszul duality $K_{i,2}(X,L)\cong K_{i+1,1}(X,L)^{\vee}$, in order to conclude, it suffices to show that $K_{i,2}(X,L)=0$. Using the main result of \cite{V2},
we have that $K_{i,2}(X,L)\cong K_{i,2}(C,\omega_C)=0$, for the genus $g$ curve $C\in |L|$ is known to be Brill-Noether general, in particular it has maximal Clifford index $\mbox{Cliff}(C)=i+1$.
\end{proof}

\vskip 5pt

In what follows, we shall repeatedly use that if $E$ is a vector bundle of rank $r$ on a stack $X$, then
\begin{equation}\label{symchern}
c_1\bigl(\bigwedge^n E\bigr)={r-1\choose n-1} c_1(E)\ \ \mbox{ and
} \ \ c_1\bigl(\mbox{Sym}^n(E)\bigr)={r+n-1\choose r} c_1(E).
\end{equation}

\begin{theorem}\label{koszclass}
Set $g=2i+3$. The class of the Koszul divisor of $K3$ surfaces with extra syzygies is given by
$$[\mathfrak{Kosz}_g]=\frac{2}{i+2}{2i-1\choose i}\Bigl(2(i+1)(i+5)\lambda+\gamma\Bigr)+\alpha\cdot [D_{1,1}]\in CH^1(\cF_g),$$
for some coefficient $\alpha\in \mathbb Z$.
\end{theorem}
\begin{proof}
As explained, off the divisor $D_{1,1}$, the locus $\mathfrak{Kosz}_g$ is the degeneracy locus of the morphism $\phi:\H_{i,2}\rightarrow \G_{i,2}$, therefore
$[\mathfrak{Kosz}_g]=c_1(\G_{i,2})-c_1(\H_{i,2})+\alpha\cdot [D_{1,1}]$, for a certain integral coefficient $\alpha$. Using repeatedly the exact sequences (\ref{exseqg}) and the formulas for the ranks of the vector bundles $\cU_{2+j}$, we find
$$c_1(\G_{i,2})=\sum_{j=0}^i (-1)^j c_1\Bigl(\bigwedge^{i-j}\cU_1\otimes \cU_{2+j}\Bigr)=
$$$$\sum_{j=0}^i (-1)^j \Bigl(\bigl(2+(j+2)^2(g-1)\bigr){g\choose i-j-1}c_1(\cU_1)+{g+1\choose i-j}c_1(\cU_{2+j})\Bigr).$$
Similarly, in order to compute the first Chern class of $\H_{i,2}$,  we use the exact sequences (\ref{exseqh}):
$$c_1(\H_{i,2})=\sum_{j=0}^i (-1)^j c_1\Bigl(\bigwedge^{i-j} \cU_1\otimes \mbox{Sym}^{j+2} \cU_1\Bigr)=$$
$$\sum_{j=0}^i (-1)^j \Bigl({g+j+2\choose g}{g\choose i-j-1}+{g+1\choose i-j}{g+j+2\choose g+1}\Bigr)c_1(\cU_1)=\frac{i+1}{2}{2i+5\choose i+2}c_1(\cU_1).$$
Substituting in these formulas the Chern classes computed in Proposition \ref{chernn1}, after some manipulations we obtain the
claimed formula for $[\mathfrak{Kosz}_g]$.
\end{proof}

\section{Lazarsfeld-Mukai bundles on $K3$ surfaces of even genus and tautological classes}\label{seclm}

For even genus, in order to obtain a Noether-Lefschetz relation between the classes $\lambda$ and $\gamma$ which is different than the one in Theorem \ref{rank4intro}, we use the geometry of the rank $2$ Lazarsfeld-Mukai vector bundle one associates to a sufficiently general polarized $K3$ surface.
We denote by $D_{\mathrm{NL}}\subseteq \cF_g$ the Noether-Lefschetz divisor consisting of $K3$ surfaces $[X,L]$ of genus $g$, such that $L=\OO_X(D_1+D_2)$, with both $D_1$ and $D_2$ being non-trivial effective divisors on $X$. We set $\cF_g^{\sharp}:=\cF_g-D_{\mathrm{NL}}$ and slightly abusing notation, we denote by $\pi:\cX^{\sharp}\rightarrow \cF_g^{\sharp}$ the corresponding restriction of the universal $K3$ surface. Throughout this subsection we fix an even genus $g=2i$, with $i\geq 4$. Our aim is to show that the restriction of both classes $\lambda$
and $\gamma$ to $\cF_g^{\sharp}$ is trivial. The geometric source of such a relation lies in the geometry of Lazarsfeld-Mukai vector bundles that have proved to be instrumental in Lazarsfeld's proof \cite{La} of the Petri Theorem.

\begin{definition}\label{lm}
For a polarized $K3$ surface $[X,L]\in \cF_g^{\sharp}$, we denote by $E_L$ the unique stable rank $2$ vector
bundle on $X$, satisfying $\mathrm{det}(E_L)=L$, $c_2(E_L)=i+1$ and $h^0(X,E_L)=i+2$.
\end{definition}

The vector bundle $E:=E_L$, which we refer to as the \emph{Lazarsfeld-Mukai vector bundle} of $[X,L]$ has been first considered in \cite{Mu} and \cite{La}. In order to construct it, one chooses
a smooth curve $C\in |L|$ and a pencil of minimal degree $A\in W^1_{i+1}(C)$. By Lazarsfeld \cite{La}, it is known that $C$ verifies the Brill-Noether Theorem, in particular $\mbox{gon}(C)=i+1$. We define the dual Lazarsfeld-Mukai bundle via the following exact sequence on $X$
\begin{equation}\label{lmseq}
0\longrightarrow E_L^{\vee} \longrightarrow H^0(C,A)\otimes \OO_X\stackrel{\mathrm{ev}}\longrightarrow \iota_*A\longrightarrow 0,
\end{equation}
where $\iota:C\hookrightarrow X$ denotes the inclusion map. Dualizing the previous sequence, we obtain the short exact sequence
$$0\longrightarrow H^0(C,A)^{\vee}\otimes \OO_X\longrightarrow E_L\longrightarrow \omega_C\otimes A^{\vee}\longrightarrow 0.$$
We summarize the properties of this vector bundle and refer to \cite{La} for proofs:
\begin{proposition}\label{lmprop}
Let $[X,L]\in \cF_g^{\sharp}$ and $E=E_L$ be the corresponding rank $2$ Lazarsfeld-Mukai bundle.
\begin{enumerate}
\item $E$ is globally generated and $H^1(X,E)=H^2(X,E)=0$.
\item $h^0(X,E)=h^0(C,\omega_C\otimes A^{\vee})+h^0(C,A)=i+2$.
\item $E$ is $\mu_L$-stable, in particular $h^0(X,E\otimes E^{\vee})=1$ as well as rigid, that is, $H^1(X,E\otimes E^{\vee})=0$.
\item The vector bundle $E$ is independent of the choice of $C$ and of that of the pencil $A\in W^1_{i+1}(C)$.
\end{enumerate}
\end{proposition}
In particular, Proposition \ref{lmprop} implies that $E$ is the only $\mu_L$-semistable sheaf on $X$ having Mukai vector $v=v(E)=(2, L, i)$.
We denote  by $\mbox{det}:\bigwedge ^2 H^0(X,E)\rightarrow H^0(X,L)$
the determinant map.


\vskip 3pt

Let $\cE$ be the universal rank 2 Lazarsfeld-Mukai vector bundle over $\cX^{\sharp}$, that is, $\cE_{|X}=E_L$, for every $[X,L]\in \cF_g^{\sharp}$. In this case $\L:=\mbox{det}(\cE)$ can be taken to be the polarization line bundle, and apart from the  classes $\kappa_{1,1}=\pi_*\bigl(c_1(\cE)\cdot c_2(\mathcal{T}_{\pi})\bigr)$ and $\kappa_{3,0}=\pi_*\bigl(c_1(\cE)^3\bigr)$, we also have a third tautological class
$$\vartheta:=\pi_*\bigl(c_1(\cE)\cdot c_2(\cE)\bigr).$$
To show that both classes $\lambda$ and $\gamma=\kappa_{1,1}-\frac{g-1}{4}\kappa_{3,0}$ are of NL type, we  need \emph{two} further sources of geometric relations in terms of Lazarsfeld-Mukai bundles. These provide two relations involving $\lambda, \kappa_{3,0}$, $\kappa_{1,1}$ and $\vartheta$, hence by eliminating $\vartheta$, one relation between $\lambda$ and $\gamma$, which turns out to be different than the one given by Theorem \ref{rank4intro}.

\vskip 4pt

\subsection{The Chow form of the Grassmannian and Lazarsfeld-Mukai bundles.}
One such source of relations is provided by the recent work \cite{AFPRW}, where among other things a Schubert-theoretic description of the Cayley-Chow form of the Grassmannian $\Gr(2,n)$ of lines is provided.
\vskip 3pt

Let $V$ be an $n$-dimensional complex vector space and choose a linear subspace $K\subseteq \bigwedge^{2} V$. Then $K^{\perp}\subseteq \bigwedge^2 V^{\vee}$. It is shown in \cite{AFPRW} Theorem 3.1 that the condition $\PP\bigl(K^{\perp})\cap \Gr(V,2)=\emptyset$, the intersection being taken inside $\PP\bigl(\bigwedge ^2 V{^\vee}\bigr)$, is equivalent to the exactness of the complex
$$K\otimes \mbox{Sym}^{n-3}(V)\stackrel{\delta_2}\longrightarrow V\otimes \mbox{Sym}^{n-2}(V)\stackrel{\delta_1}\longrightarrow \mbox{Sym}^{n-1}(V),$$
where $\delta_i:\bigwedge^i V\otimes \mbox{Sym}(V)\rightarrow \bigwedge^{i-1} V\otimes \mbox{Sym}(V)$ denotes the Koszul differential, for $i=1,2$.

\vskip 3pt

We apply this result for polarized $K3$ surfaces, when we take $V:=H^0(X,E)^{\vee}$ and $$K^{\perp}:=\mbox{Ker}\Bigl\{\mbox{det}:\bigwedge^2 H^0(X,E)\rightarrow H^0(X,L)\Bigr\}$$ is the kernel of the determinant map.
Note that $\det$ does not vanish on any element of rank $2$, see \cite{V2} page 380, for the existence of an element $0\neq s_1\wedge s_2\in \bigwedge^2 H^0(X,E)$ such that $\mbox{det}(s_1\wedge s_2)=0$, would imply a splitting of $L$ as a sum of two pencils. By dualizing, we conclude that the complex
\begin{equation}\label{complex1}
\mbox{Sym}^{i+1} H^0(X,E)\longrightarrow H^0(X,E)\otimes \mbox{Sym}^i H^0(X,E)\stackrel{\beta}\longrightarrow H^0(X,L) \otimes \mbox{Sym}^{i-1} H^0(X,E)
\end{equation}
is exact for every point $[X,L]\in \cF_g^{\sharp}$. The map $\beta$ is obtained by composing the (dual) Koszul differential
$$H^0(E)\otimes \mbox{Sym}^i H^0(X,E)\rightarrow \bigwedge^2H^0(X,E)\otimes \mbox{Sym}^{i-1}H^0(X,E)$$ with the map $\mbox{det}\otimes \mbox{id}_{\mathrm{Sym}^{i-1} H^0(X,E)}:\bigwedge ^2 H^0(X,E)\otimes \mbox{Sym}^{i-1} H^0(X,E)\rightarrow H^0(X,L)\otimes \mbox{Sym}^{i-1} H^0(X,E)$.

\vskip 4pt

We globalize this geometric fact. For $n\geq 1$, we introduce the vector bundle
$\cV_n:=\pi_*\bigl(\mbox{Sym}^n\cE\bigr)$ on $\cF_g^{\sharp}$, where we observe that $R^i\pi_*\bigl(\mbox{Sym}^n\cE\bigr)=0$, for $i=1,2$.
We shall make use of the following formulas:

\begin{proposition}\label{chernlm}
The following formulas hold in $CH^1(\cF_{2i}^{\sharp})$:
$$c_1(\cV_1)=\frac{1}{12}\kappa_{1,1}+\frac{1}{6}\kappa_{3,0}-\frac{i+2}{2}\lambda-\frac{1}{2}\vartheta\  \ \mbox{ and }  \ \ c_1(\cV_2)=\frac{1}{4}\kappa_{1,1}
+\frac{3}{2}\kappa_{3,0}-\frac{6i-3}{2}\lambda-4\vartheta.$$
\end{proposition}
\begin{proof}
We only discuss the calculation of $c_1(\cV_2)$. For any $[X,L]\in \cF_g^{\sharp}$, observe that
$$h^0\bigl(X,\mbox{Sym}^2(E)\bigr)=\chi\bigl(X,\mbox{Sym}^2(E)\bigr)=\frac{c_1^2\bigl(\mbox{Sym}^2 (E)\bigr)}{2}-c_2(\mbox{Sym}^2(E))+3\chi(X,\OO_X)=6i-3,$$
where we use the formulas $c_1\bigl(\mbox{Sym}^2(E)\bigr)=3c_1(E)$ and $c_2\bigl(\mbox{Sym}^2(E)\bigr)=2c_1^2(E)+4c_2(E)$. Applying Grothendieck-Riemann-Roch to the universal family
$\pi:\cX^{\sharp}\rightarrow \cF_g^{\sharp}$, we find:
$$c_1(\cV_2)=c_1\bigl(\pi_{!}\bigl(\mbox{Sym}^2 \cE\bigr))=\pi_*\Bigr[\Bigl(3+3c_1(\cE)+\frac{5c_1^2(\cE)-8c_2(\cE)}{2}+\frac{9c_1^3(\cE)-24c_1(\cE)c_2(\cE)}{6}\Bigr)\cdot $$$$\Bigl(1-\frac{1}{2}c_1(\omega_{\pi})+\frac{1}{12}\bigl(c_1^2(\omega_{\pi})+c_2(\Omega_{\pi})\bigr)-
\frac{1}{24} c_1(\omega_{\pi})c_2(\Omega_{\pi})+\cdots \Bigr)\Bigr]_2.$$
Expanding the product and using again that $\pi_*(c_2(\Omega_{\pi}))=24$, we obtain the claimed formula.
\end{proof}

\vskip 3pt

In order to treat the complex (\ref{complex1}) variationally, we consider the following vector bundles over $\cF_{2i}^{\sharp}$

$$\cA:=\frac{\cV_1\otimes \mbox{Sym}^i(\cV_1)}{\mbox{Sym}^{i+1}(\cV_1)} \ \ \ \mbox{ and } \ \ \ \cB:=\cU_1\otimes \mbox{Sym}^{i-1}(\cV_1).$$

Note that $\mbox{rk}(\cA)=(i+2){2i+1\choose i}-{2i+2\choose i+1}=(2i+1){2i\choose i-1}=\mbox{rk}(\cB)$ and there is a sheaf morphism
$$\beta:\cA\rightarrow \cB,$$ which over a point $[X,L]\in \cF_g^{\sharp}$ is precisely the map
$$\beta_{X,L}:\frac{H^0(X,E)\otimes \mbox{Sym}^iH^0(X,E)}{\mbox{Sym}^{i+1} H^0(X,E)}\rightarrow H^0(X,L)\otimes \mbox{Sym}^{i-1} H^0(X,E)$$
induced by (\ref{complex1}). As explained, the morphism $\beta$ is everywhere non-degenerate over $\cF_{g}^{\sharp}$.

\begin{theorem}\label{thetaformula}
One has the following formula
$$\vartheta=\frac{i}{8i+4}\kappa_{1,1}+\frac{i}{4i+2}\kappa_{3,0}-\frac{i+2}{2}\lambda.$$
\end{theorem}
\begin{proof}
The morphism $\beta:\cA\rightarrow \cB$ being everywhere non-degenerate, we find that $c_1(\cA)=c_1(\cB)$. Applying systematically the formulas (\ref{symchern}), we write:
$$c_1(\cA)=\Bigl({2i+1\choose i}+(i+2){2i+1\choose i+2}-{2i+2\choose i+2}\Bigr)c_1(\cV_1), \ c_1(\cB)={2i\choose i-1}c_1(\cU_1)+(2i+1){2i\choose i+2}c_1(\cV_1),$$
hence, after manipulations
$$0=c_1(\cB-\cA)={2i\choose i-1}\Bigl(c_1(\cU_{1})-\frac{4i+2}{i+2}c_1(\cV_1)\Bigr).$$
We then replace $c_1(\cU_1)$ and $c_1(\cV_1)$ with their respective expressions provided by Propositions \ref{chernn1} and \ref{chernlm}, clear denominators (our Chow groups are with $\mathbb Q$-coefficients), then conclude.
\end{proof}

\subsection{Lazarsfeld-Mukai bundles and rank $6$ quadrics.} A second source of relations between the classes $\lambda$, $\kappa_{1,1}$, $\kappa_{3,0}$ and
$\vartheta$ is obtained by studying the kernel the multiplication map
$$\mu_E:\mbox{Sym}^2 H^0(X,E)\rightarrow H^0\bigl(X,\mbox{Sym}^2(E)\bigr)$$
associated to the Lazarsfeld-Mukai bundle $E=E_L$ corresponding to an element $[X,L]\in \cF_{2i}^{\sharp}$. We assume throughout that $i\geq 4$.

\begin{lemma}\label{symvan}
One has $H^i\bigl(X, \mathrm{Sym}^2(E)\bigr)=0$, for $i=1,2$.
\end{lemma}
\begin{proof}
We choose a general curve $C\in |L|$ and a minimal pencil $A\in W^1_{i+1}(C)$. Tensoring the exact sequence (\ref{lmseq}) by $E^{\vee}$ and taking global sections implies $H^2\bigl(X, E\otimes E\bigr)\cong H^0(X, E^{\vee}\otimes E^{\vee})=0$.
Similarly, we can prove that $H^1(X,E\otimes E)=0$, which implies $H^1\bigl(X,\mbox{Sym}^2(E)\bigr)=0$. We tensor again (\ref{lmseq}) by $E^{\vee}$ and take cohomology.
The vanishing of $H^1(X,E^{\vee}\otimes E^{\vee})\cong H^1(X,E\otimes E)^{\vee}$ is implied by $H^0(C,A\otimes E_{|C}^{\vee})=0$, which follows because $E_{|C}$ is stable on $C$  and
$\mu\bigl(A\otimes E_{|C}^{\vee}\bigr)=4-2i<0$.
\end{proof}

\vskip 3pt

Using Lemma \ref{symvan}, we compute $h^0\bigl(X,\mbox{Sym}^2(E)\bigr)=6i-3$, then observe that
$$\mbox{dim } \mbox{Sym}^2 H^0(X,E)-h^0\bigl(X,\mbox{Sym}^2(E)\bigr)={i+3\choose 2}-(6i-3)={i-3\choose 2},$$
that is, the locus
$$D_{2i}^{\mathrm{rk}6}:=\Bigl\{[X,L]\in \cF_{2i}^{\sharp}: \exists\ 0\neq q\in \mbox{Ker}(\mu_E), \ \mbox{rk}(q)\leq 6\Bigr\},$$
is expected to be a divisor on $\cF_{2i}^{\sharp}$. We confirm this expectation in a very precise form.

\begin{theorem}\label{rk6}
For a polarized $K3$ surface $[X,L]\in \cF_{2i}^{\sharp}$, the kernel of the map $\mu_E$ contains no non-zero elements of rank at most $6$, that is, $D_{2i}^{\mathrm{rk}6}=\emptyset$.
\end{theorem}
\begin{proof}
We start with a $K3$ surface $[X,L]\in \cF_{2i}^{\sharp}$ and assume we have an element $0\neq q\in \mbox{Ker}(\mu_E)$, where $\mbox{rk}(q)=n\leq 6$. We write
$q=s_1^2+\cdots+s_n^2$, where $s_i\in H^0(X,E)$. Denoting by $\PP(E)\rightarrow X$ the projective bundle associated to $E$, we have the canonical identifications
$$H^0(X, E)\cong H^0\bigl(\PP(E),\OO_{\PP(E)}(1)\bigr) \ \mbox{ and } \ H^0\bigl(X,\mbox{Sym}^2(E)\bigr)\cong H^0\bigl(\PP(E),\OO_{\PP(E)}(2)\bigr).$$
Let $V:=\langle s_1, \ldots, s_n\rangle\subseteq H^0(X,E)$. Since $H^0(X,E(-L))=0$,  the cokernel of the evaluation map $V\otimes \OO_X\rightarrow E$
is supported along finitely many points and we denote by $$\varphi_V:X\dashrightarrow \Gr(V,2)\subseteq \PP\bigl(\bigwedge^2 V^{\vee}\bigr)$$ the induced rational map.
Note that $\varphi_V^*(\OO(1))=L$. We further denote by $Q\subseteq \PP\bigl(V^{\vee}\bigr)$ the quadric given by the equation $q=0$. The condition $q\in \mbox{Ker}(\mu_E)$ can be interpreted geometrically as saying that the image of $\PP(E)$ under the composition $$\PP(E)\stackrel{|\OO(1)|}\longrightarrow \PP\bigl(H^0(E)^{\vee})\dashrightarrow \PP(V^{\vee})$$ lies on the quadric $Q$. This in turn, amounts to saying that $\varphi_V(X)$ is contained in orthogonal Grassmannian $\Gr_Q\subseteq \Gr(V,2)$ of lines in $\PP\bigl(V^{\vee})$ contained in $Q$. The essential observation is that for $n\leq 6$, the pull-back $\OO_{\Gr_Q}(1)$ splits \emph{non-trivially} as the sum of two effective line bundles, which in turn, induces a decomposition of the polarization class $L$ on $X$, contradicting the assumption that $[X,L]$ is NL-general.

\vskip 3pt

We discuss in detail the case $n=6$, the situation for $n\leq 5$ being similar. Thus $Q\subseteq \PP\bigl(V^{\vee}\bigr)=\PP^5$ is a rank $6$ quadric and we may assume that
$Q=\Gr(2,U)\subseteq \PP\bigl(\bigwedge^2 U\bigr)\cong \PP(V^{\vee})$ is the Grassmannian of lines in $\PP^3\cong \PP(U)$, where $U$ is a $4$-dimensional complex vector space such that $\bigwedge^2 U\cong V^{\vee}$.
Then every line inside $\Gr_Q$ is of the form $L_{\ell,H}:=\bigl\{\Pi\in \Gr(2,U): \ell \subseteq \Pi\subseteq H\bigr\}$, where $\ell\subseteq U$ is $1$-dimensional and $H\subseteq U$ is  a $3$-dimensional subspace. Accordingly, one has an isomorphism between $\Gr_Q$ and the incidence correspondence  $\Sigma\subseteq \PP(U)\times \PP(U^{\vee})$, assigning to the pair $(\ell, H)\in \Sigma$ with $\ell\subseteq  H$ the line $L_{\ell,H}$ defined above. Denoting by
$$
\begin{CD}
\PP(U) @<\pi_1<< \Gr_Q\cong \Sigma \subseteq \PP(U)\times \PP(U^{\vee}) @>\pi_2>> \PP(U^{\vee}) \\
\end{CD}
$$
the two projections, we have $\OO_{\Gr_Q}(1)\cong \pi_1^*\bigl(\OO_{\PP(U)}(1)\bigr)\otimes \pi_2^*\bigl(\OO_{\PP(U^{\vee})}(1)\bigr).$
Let $q_1:=\pi_1\circ \varphi_V:X\dashrightarrow \PP(U)$ and $q_2:=\pi_2\circ \varphi_V:X\dashrightarrow \PP(U^{\vee})$. It is now enough to observe that
$h^0(X,q_i^*(\OO(1))\geq 2$, for $i=1,2$. Indeed, else the image of one of the maps $q_i$, say $q_1$, is a point, hence there exists $\ell_0\in \PP(U)$, such that
$\mbox{Im}(\varphi_V)\subseteq \pi^{-1}_1(\ell_0)$. But $\pi^{-1}_1(\ell_0)\cong \PP^2\subseteq \PP\bigl(\bigwedge^2 U\bigr)$, that is, $\varphi_V(X)\subseteq \Gr(2,3)$, which is impossible, for $V$ is $6$-dimensional. We conclude that $L=q_1^*\bigl(\OO_{\PP(U)}(1)\bigr)\otimes q_2^*\bigl(\OO_{\PP(U^{\vee})}(1)\bigr)$ is NL special.

\vskip 3pt

We briefly mention the cases $n\leq 5$. For $n=4$, we have $Q\subseteq \PP^3$ and the variety of lines $\Gr_Q$ consists of two copies of $\PP^1$. For $n=5$, when $Q\subseteq \PP^4$ is a rank $5$ quadric, the variety of lines $\Gr_Q$ is identified with $\PP^3$ in its second Veronese embedding $\PP^3\hookrightarrow \PP^9\cong \PP\bigl(\bigwedge^2 V^{\vee}\bigr)$. The assumption that there exist $0\neq q\in \mbox{Ker}(\mu_E)$ with $\mbox{rk}(q)\in \{4,5\}$ implies that $L$ is non-primitive,  a contradiction.
\end{proof}

\begin{theorem}\label{secondrelation}
The relation $\frac{1}{2i+1}\gamma+(i+2)\lambda=0$ holds in $CH^1(\cF_{2i}^{\sharp})$.
\end{theorem}
\begin{proof} We first use the fact that the divisor $D_{2i}^{\mathrm{rk}6}$ is of Noether-Lefschetz type, that is, by applying Theorem \ref{classdiv1} coupled with proposition \ref{chernlm}, we obtain the relation
$$0=[D_{2i}^{\mathrm{rk}6}]=c_1(\cV_2)-\frac{2(6i-3)}{i+2}c_1(\cV_1)=\frac{3}{2}(2i-1)\lambda+\frac{2i-11}{i+2}\vartheta-\frac{i-8}{2i+4}\kappa_{3,0}-\frac{3i-4}{4i+8}\kappa_{1,1}.$$
Substituting $\vartheta$ obtained in this way in the formula provided by Theorem \ref{thetaformula}, we obtained the desired relation between $\lambda$ and $\gamma$.
\end{proof}

\vskip 4pt

\noindent
{\emph{Proof of Theorem \ref{borcherds} for even $g$.} Theorem \ref{rank4intro} provides the relation $(4i-1)\lambda+\frac{2}{2i+1}\gamma=0\in CH^1(\cF_{2i}^{\sharp})$. Coupled with Theorem \ref{secondrelation}, we conclude that both classes $\lambda$ and $\gamma$ vanish on $\cF_{2i}^{\sharp}$, hence they are of Noether-Lefschetz type on $\cF_{2i}$.
\hfill $\Box$

\section{Semistability of the second Hilbert point of a polarized $K3$ surface}\label{gitss}

A simple and somewhat surprising application of the techniques developed in Subsection \ref{k34sub} concerns the GIT semistability of polarized $K3$ surfaces. We fix
once and for all a vector space $V\cong \mathbb C^{g+1}$.

\begin{definition} For
a polarized $K3$ surface $[X,L]\in \cF_g$ such that $\mathrm{Pic}(X)=\mathbb Z\cdot L$, we define its \emph{second Hilbert point} to be the quotient
$$[X,H]_2:=\Bigl[\mathrm{Sym}^2 H^0(X,L)\longrightarrow H^0(X,L^{\otimes 2})\longrightarrow 0\Bigr] \in \Gr\Bigl(\mathrm{Sym}^2 H^0(X,H), 4g-2\Bigr).$$
\end{definition}

The group $SL(V)$ acts linearly  on the Grassmannian $\Gr\bigl(\mbox{Sym}^2 V, 4g-2\bigl)\subseteq \PP\Bigl(\bigwedge^{4g-2} \mbox{Sym}^2 V^{\vee}\Bigr)$, where the last inclusion is given by the Pl\"ucker embedding. Let $\overline{\mathrm{Hilb}}_g$ be the closure inside the quotient
$$\PP\Bigl(\bigwedge^{4g-2} \mbox{Sym}^2 V^{\vee}\Bigr)\dblq SL(V)$$
of the locus of semistable second Hilbert points $[X,H]_2$ of quasi-polarized $K3$ surfaces of genus $g$.
Then the GIT quotient
$$\overline{\cF}_g:=\overline{\mathrm{Hilb}_g^{\mathrm{ss}}}\dblq SL(V)$$
is a projective birational model of the moduli space $\cF_g$, provided the locus $\overline{\mathrm{Hilb}_g^{\mathrm{ss}}}$ of semistable 2nd Hilbert points is not empty.
\begin{theorem}\label{git}
Let $[X,L]\in \cF_g$ be a polarized $K3$ surface with $\mathrm{Pic}(X)\cong \mathbb Z\cdot L$. Then the second Hilbert point $[X,H]_2$ is semistable.
In particular, $\overline{\cF}_g$, defined as above, exists.
\end{theorem}

\begin{proof}By definition of semistability, since the Grassmannian $\Gr(\mbox{Sym}^2 V,4g-2)$ has Picard numer $1$, it suffices to construct an $SL(V)$-invariant effective divisor $\mathcal{D}$ of $\Gr(\mbox{Sym}^2 V,4g-2)$ such that $[X,L]_2\notin \mathcal{D}$. Theorem \ref{rank4intro} provides such a divisor. We take $\mathcal{D}$ to be the locus of
$(4g-2)$-dimensional quotients $\phi:\mbox{Sym}^2 V\twoheadrightarrow Q$ such that $\mbox{Ker}(\phi)$ contains a quadric of rank at most $4$. The parameter count from Subsection
\ref{k34sub} shows that $\mathcal{D}$ is indeed a divisor. If $[X,H]\notin D_g^{\mathrm{rk} 4}$, then $I_{X,L}(2)$ contains no quadrics of rank at most $4$, in particular
$[X,L]_2\notin \mathcal{D}$, hence its second  Hilbert point is semistable.
\end{proof}
\begin{remark} By the analogy with the much studied case of the moduli space of curves \cite{HH}, we expect that, apart from smooth $K3$ surfaces, $\overline{\cF}_g$ also parametrizes degenerate $K3$ surfaces with various singularities. It is also likely that the 2nd Hilbert point of NL special smooth K3 surfaces is not semistable, that is, the natural map $\cF_g\dashrightarrow \overline{\cF}_g$ might not be regular along NL special loci.
\end{remark}

\section{The geometry of the Hurwitz spaces of admissible covers}\label{hurw1}

In what follows, for a Deligne-Mumford stack  $M$, we shall denote by $\cM$ its coarse moduli space. If $X\subseteq \cM$ is an irreducible subvariety, we denote by $[X]\in CH^*_{\mathbb Q}(\cM)$ its class in the stack sense, that is, we divide by the order of the automorphism group of a general element in $X$.

We denote by  $\mathcal{H}_k^{\mathfrak o}$ the Hurwitz space of degree $k$ covers $f:C\rightarrow \PP^1$ with simple ramification from a smooth curve $C$ of genus $2k-1$, together with an \emph{ordering} $(p_1, \ldots, p_{6k-4})$ of its branch points.
Let $\hh_k^{\mathfrak o}$ denote the compactification of $\mathcal{H}_k^{\mathfrak o}$ by admissible covers. By \cite{ACV}, the stack $\overline{H}_k^{\mathfrak o}$ (whose coarse moduli space is precisely $\hh_k^{\mathfrak o}$) is isomorphic to the stack of \emph{twisted stable} maps into the classifying stack $\mathcal{B} \mathfrak{S}_k$ of the symmetric group $\mathfrak S_k$, that is,
$$\overline{H}_k^{\mathfrak o}:=\overline{M}_{0,6k-4}\Bigl(\mathcal{B} \mathfrak S_k\Bigr).$$
Points of $\hh_k^{\mathfrak o}$ are admissible covers
$[f:C\rightarrow R, p_1, \ldots, p_{6k-4}]$, where $C$ and $R$ are nodal curves of genus $2k-1$ and $0$ respectively, and $p_1, \ldots, p_{6k-4}\in R_{\mathrm{reg}}$
are the branch points of $f$.  Let
$\mathfrak{b}:\hh_k^{\mathfrak o}\rightarrow \mm_{0, 6k-4}$ be the \emph{branch} morphism. The symmetric group
$\mathfrak S_{6k-4}$ acts on $\hh_k^{\mathfrak o}$ by permuting the branch points of each cover. Denoting by
$$\hh_k:=\hh_k^{\mathfrak o}/\mathfrak S_{6k-4}$$ the quotient parametrizing admissible covers without an ordering of the branch points, the projection $q:\hh_k^{\mathfrak o} \rightarrow \hh_k$ is a principal $\mathfrak S_{6k-4}$-bundle. We denote by
$ \sigma:\hh_k\rightarrow \mm_{2k-1}$
the map assigning to an admissible cover the stable model of its source curve. We shall use throughout the isomorphism $CH^1_{\mathbb Q}(\hh_k)\cong H^2(\hh_k, \mathbb Q)$, see \cite{DE} Theorem 5.1 and Proposition 2.2.

\vskip 3pt

For $i=0, \ldots, 3k-2$, let $B_i$ be the boundary divisor of $\mm_{0,6k-4}$ whose general point is  the union of two smooth rational curves meeting at one point, such that precisely $i$ of the marked points lie on one component. The boundary divisors of $\hh_k^{\mathfrak o}$ are parametrized by the following combinatorial data:
\begin{enumerate}
\item A partition $I\sqcup J=\{1,\dotsc,6k-4\}$, such that $|I|\geq 2$,
$|J|\geq 2$.
\item Transpositions $\{w_i\}_{i\in I}$ and $\{w_j\}_{j\in J}$ in $\mathfrak S_k$, with $\prod_{i\in I} w_i = u$, $\prod_{j\in J}w_j=u^{-1}$, for some
$u\in \mathfrak S_k$.
\end{enumerate}

To this data, we associate the locus of admissible covers with labeled branch points
$$\bigl[f: C\rightarrow R, \ p_1, \ldots, p_{6k-4}\bigr]\in \hh_k^{\mathfrak o},$$
where $[R=R_1\cup_p R_2, p_1, \ldots, p_{6k-4}]\in B_{|I|}\subseteq \mm_{0,6k-4}$ is a pointed union of two smooth rational curves $R_1$ and $R_2$ meeting at the point $p$. The marked points lying on $R_1$ are precisely those labeled by the set $I$. Let $\mu:=(\mu_1, \ldots, \mu_{\ell})\vdash k$ be the partition induced by $u\in \mathfrak S_k$ and denote by $E_{i:\mu}$ the boundary divisor on $\hh_k^{\mathfrak o} $ classifying twisted stable maps with underlying admissible cover as above, with $f^{-1}(p)$ having partition type $\mu$, and precisely $i$ of the points $p_1, \ldots, p_{6k-4}$ lying on $R_1$. We denote by $D_{i:\mu}$ the reduced boundary divisor of $\hh_k$ pulling back to $E_{i:\mu}$ under the map $q$.

For $i=2, \ldots, 3k-2$, we have the following relation, see \cite{HM} p. 62, as well as \cite{GK1} Lemma 3.1:
\begin{equation}\label{pb}
\mathfrak{b}^*(B_i)=\sum_{\mu\vdash k} \mbox{lcm}(\mu)E_{i:\mu}.
\end{equation}

\vskip 3pt
The class of the Hodge class $\lambda:=(\sigma\circ q)^*(\lambda)$ on $\hh_k^{\mathfrak o}$ has been determined in \cite{KKZ} and \cite{GK1}:
\begin{equation}\label{hodgecl}
\lambda=\sum_{i=2}^{3k-2} \sum_{\mu\vdash k} \mathrm{lcm}(\mu)\Bigl(\frac{i(6k-4-i)}{8(6k-5)}-\frac{1}{12}\Bigl(k-\sum_{j=1}^{\ell(\mu)} \frac{1}{\mu_j}\Bigr)\Bigr)[E_{i:\mu}]\in CH^1(\hh_k^{\mathfrak{o}}).
\end{equation}

The sum is taken over partitions $\mu$ that correspond to permutations that can be written as products of $i$ transpositions. Furthermore, $\ell(\mu)$ denotes the length of the partition $\mu$ and $\mbox{lcm}(\mu)$ is the lowest common multiple of the parts of $\mu$.

\vskip 4pt

We now discuss in detail the divisors of $\hh_k^{\mathfrak {o}}$ lying over the boundary divisor $B_2$. We pick a cover $$[f:C=C_1\cup C_2\rightarrow R=R_1\cup_p R_2,\  p_1, \ldots, p_{6k-4}]\in \mathfrak{b}^*(B_2),$$ where $C_i=f^{-1}(R_i)$. We assume $I=\{1, \ldots, 6k-6\}$, thus $p_1, \ldots, p_{6k-6}\in R_1$ and $p_{6k-5}, p_{6k-4}\in R_2$.

We denote by $E_{2:(1^k)}$ the closure in $\hh_k^{\mathfrak o}$ of the set of admissible covers for which the transpositions $w_{6k-5}$ and $w_{6k-4}$ corresponding to the points $p_{6k-5}$ and $p_{6k-4}$ are equal. Let $E_0$ be the component of $E_{2:(1^k)}$ corresponding to the case when $C_1$ is connected, which happens precisely when
$\langle w_1, \ldots, w_{6k-6}\rangle =\mathfrak S_k$. Clearly $E_0$ is irreducible. When $w_{6k-5}$ and $w_{6k-4}$ are distinct but not disjoint then $\mu=(3,1^{k-3})\vdash k$ and one is led to the boundary divisor $E_{2: (3,1^{k-3})}$. We denote by $E_3$ the (irreducible) subdivisor of $E_{2:(3,1^{k-3})}$ corresponding to the case $\langle w_1, \ldots, w_{6k-6}\rangle =\mathfrak S_k$. Finally, the case when $w_{6k-5}$ and $w_{6k-4}$ are disjoint corresponds to the boundary divisor $E_{2:(2,2,1^{k-4})}$ and we denote by $E_2$ the irreducible component of $E_{2:(2,2,1^{k-4})}$ parametrizing covers for which  $\langle w_1, \ldots, w_{6k-6}\rangle =\mathfrak S_k$.

\vskip 4pt

We discuss the behavior of the  divisors $E_0, E_2$ and $E_3$ under the map $q$ and we have
$$q^*(D_0)=2E_0, \ q^*(D_2)=E_2 \ \mbox{ and } q^*(D_3)=2E_3.$$
Indeed the general point of both $E_0$ and $E_3$ has no automorphism that fixes all branch points, but admits an automorphism of order two that fixes $C_1$ and permutes the branch points $p_{6k-4}$ and $p_{6k-5}$. The general admissible cover  in $E_2$ has an automorphism group $\mathbb Z_2\times \mathbb Z_2$ (each of the two components of $C_2$ mapping $2:1$ onto $R_2$ has an automorphism of order $2$). In the stack $\overline{H}_k^{\mathfrak o}$ we have two points lying over this admissible cover and each of them has an automorphism group of order $2$. In particular the map $\overline{H}_k^{\mathfrak o}\rightarrow \hh_k^{\mathfrak o}$ from the stack to its coarse moduli space is simply ramified along $E_2$.

\vskip 4pt

The Hurwitz formula applied to the finite ramified cover $\mathfrak{b}:\hh_k^{\mathfrak o}\rightarrow \mm_{0, 6k-4}$, coupled with the expression
$K_{\mm_{0, 6k-4}}= \sum_{i=2}^{3k-2} \bigl(\frac{i(6k-4-i)}{6k-5}-2\bigr)[B_i]$ for the canonical class of $K_{\mm_{0,6k-4}}$, yields the following formula (on the stack!):
\begin{equation}\label{canhur}
K_{\overline{H}_k^{\mathfrak o}}=\mathfrak{b}^*K_{\mm_{0, 6k-4}}+\mbox{Ram}(\mathfrak{b}),
\end{equation}
where $\mbox{Ram}(\mathfrak{b})=\sum_{i,\mu} (\mbox{lcm}(\mu)-1)E_{i:\mu}$.

\subsection{A partial compactification of the Hurwitz space}
It turns out to be convenient to work with a partial compactification of $\H_k$. We denote by $\widetilde{\H}_k$ the (quasi-projective) parameter space of pairs $[C,A]$, where $C$ is an irreducible nodal curve of genus $2k-1$ and $A$ is a base point free locally free sheaf of degree $k$ on $C$ with $h^0(C,A)=2$. There exists a rational map $\hh_k\dashrightarrow \widetilde{\H}_k$, well-defined on the general point of each of the divisors $D_0, D_2$ and $D_3$ respectively. Under this map, to the general point $[f:C_1\cup C_2\rightarrow R_1\cup_p R_2]$ of $D_3$ (respectively $D_2$) we assign the pair $[C_1, A_1:=f^*\OO_{R_1}(1)]\in \widetilde{\H}_k$. Note that $C_1$ is a smooth curve of genus $2k-1$ and $A_1$ is a pencil with a triple point (respectively with two ramification points in the fibre over $p$).  The two partial compactifications $\H_k\cup D_0\cup D_2\cup D_3$ and $\widetilde{\H}_k$ differ outside a set of codimension at least $2$ and for divisor class calculations they will be identified. Using this, formula (\ref{hodgecl}) has the following translation at the level of $\widetilde{\H}_k$:
\begin{equation}\label{hodgepart}
\lambda=\frac{3(k-1)}{4(6k-5)}[D_0]-\frac{1}{4(6k-5)}[D_2]+\frac{3k-7}{12(6k-5)}[D_3]\in CH^1(\widetilde{\H}_k).
\end{equation}

We now record the formula for the canonical class of $\widetilde{\H}_k$:

\begin{proposition}\label{canpart2}
One has $K_{\widetilde{\H}_k}=8\lambda+\frac{1}{6}[D_3]-\frac{3}{2}[D_0]$.
\end{proposition}
\begin{proof} We combine the equation (\ref{canhur}) with the Hurwitz formula applied to $q:\hh_{k}^{\mathfrak o}\dashrightarrow \widetilde{\H}_k$ and write:
$$q^*(K_{\widetilde{\H}_k})=K_{\overline{H}_k^{\mathfrak o}}-[E_0]-[E_2]-[E_3]=q^*\Bigl(-\frac{2}{6k-5}[D_2]-\frac{6k-3}{2(6k-5)}[D_0]+\frac{6k-11}{2(6k-5)}[D_3]\Bigr).$$
The divisors $E_0$ and $E_3$ lie in the ramification locus of $q$, so they are subtracted from $K_{\overline{H}_k^{\mathfrak 0}}$. Furthermore, the morphism $\overline{H}_k^{\mathfrak{o}}\rightarrow \hh_k^{\mathfrak o}$ is simply ramified along $E_2$, so this divisor has to be subtracted as well.
We now use (\ref{hodgepart}) to express  $[D_2]$ in terms of $\lambda$, $[D_0]$ and $[D_3]$ and reach the claimed formula.
\end{proof}

\vskip 4pt

Let $f:\cC_{k}\rightarrow \widetilde{\H}_{k}$ be the universal curve and we choose a universal degree $k$ line bundle $\mathcal{L}$ on $\cC_k$, that is, satisfying $\L_{| f^{-1}[C,A]}=A$, for every $[C,A]\in \widetilde{\H}_k$.  Just like in Section \ref{smallslopesec}, we define the following codimension one tautological classes:
$$\mathfrak{a}:=f_*\bigl(c_1^2(\mathcal{L})\bigr) \mbox{ } \mbox{and  } \mbox{ } \mathfrak{b}:=f_*\bigl(c_1(\mathcal{L})\cdot c_1(\omega_f)\bigr)\in  CH^1(\widetilde{\H}_k).$$
Note that $\mathcal{V}:=f_*\mathcal{L}$ is a vector bundle of rank two on $\widetilde{\H}_k$. Although  $\mathcal{L}$ is not unique, the class
\begin{equation}\label{gamma}
\gamma:=\mathfrak b-\frac{2k-2}{k}\mathfrak a\in CH^1(\widetilde{\H}_k)
\end{equation}
is well-defined and independent of the choice of $\mathcal{L}$.
\begin{proposition}\label{basepoint}
We have that $\mathfrak{a}=kc_1(\cV)\in CH^1(\widetilde{\H}_{k})$.
\end{proposition}
\begin{proof} Recall that $\widetilde{\H}_{k}$ has been defined to consist of pairs $[C,A]$ such that $A$ is a base point free pencil of degree $k$. In particular, the image under $f$ of the codimension $2$  locus in $\mathcal{C}_k$ where the morphism of vector bundles  $f^*(\cV)\rightarrow \mathcal{L}$ is not surjective is empty, hence by Porteous' formula
$$0=f_*\Bigl(c_2(f^*\cV)-c_1(f^*\cV)\cdot c_1(\mathcal{L})+c_1^2(\mathcal{L})\Bigr)=-kc_1(\cV)+\mathfrak{a}.$$
\end{proof}
We now introduce the following locally free sheaves on $\widetilde{\H}_k$:
$$\cE:=f_*(\omega_f\otimes \L^{\vee}) \ \mbox{ and } \cF:=f^*(\omega_f^{2}\otimes \L^{-2})$$
\begin{proposition}\label{chernhurw1}
The following formulas hold
$$c_1(\cE)=\lambda-\frac{1}{2}\mathfrak b+\frac{k-2}{2k}\mathfrak a \ \mbox{ and } \ c_1(\cF)=13\lambda+2\mathfrak{a}-3\mathfrak{b}-[D_0].$$
\end{proposition}
\begin{proof} We apply Grothendieck-Riemann-Roch twice to $f$. Since $R^1f_*(\omega_f^2\otimes \L^{-2})=0$, we write:
$$c_1(\cF)=f_*\Bigl[\Bigl(1+2c_1(\omega_f)-2c_1(\mathcal{L})+2(c_1(\omega_f)-c_1(\mathcal{L})^2\Bigr)\cdot \Bigl(1-\frac{c_1(\omega_f)}{2}+\frac{c_1^2(\Omega^1_f)+c_2(\Omega^1_f)}{12}\Bigr)\Bigr]_2.$$
Now use Proposition \ref{basepoint} as well as $f_*\bigl(c_1^2(\Omega_f^1)+c_2(\Omega_f^1)\bigr)=12\lambda$, see \cite{HM} p. 49, in order  to conclude.
To compute the first Chern class of $\cE$, note that $c_1\bigl(R^1f_*(\omega_f\otimes \mathcal{L}^{\vee})\bigr)=-c_1(\cV)$, hence applying again Grothendieck-Riemann-Roch together with Proposition
\ref{basepoint}, we write:
$$c_1(\cE)=-kc_1(\cV)+f_*\Bigl[\Bigl(1+c_1(\omega_f)-c_1(\mathcal{L})+\frac{(c_1(\omega_f)-c_1(\mathcal{L})^2}{2}\Bigr)\cdot \Bigl(1-\frac{c_1(\omega_f)}{2}+\frac{c_1^2(\Omega^1_f)+c_2(\Omega^1_f)}{12}\Bigr)\Bigr]_2,$$
which quickly leads to the claimed formula.

\end{proof}

We summarize the relation between the class $\gamma$ and the classes $[D_0], [D_2]$ and $[D_3]$ as follows:

\begin{proposition}\label{trans}
One has the formula $[D_3]=6\gamma+24\lambda-3[D_0].$
\end{proposition}
\begin{proof}
 We form the fibre product of the universal curve $f:\mathcal{C}_{k}\rightarrow \widetilde{\H}_k$ together with its projections:
$$
\begin{CD}
{\mathcal{C}_{k}} @<\pi_1<< {\mathcal C_{k}\times_{\widetilde{\H}_{k}}\mathcal{C}_{k}} @>\pi_2>> {\mathcal{C}_{k}} \\
\end{CD}.
$$
For $\ell\geq 1$, we introduce the sheaf of principal parts $\mathcal P_{f}^{\ell}(\L):=(\pi_2)_{*}\Bigl(\pi_1^*(\mathcal{L})\otimes \mathcal{I}_{(\ell+1)\Delta}\Bigr)$.
Observe that $\mathcal P_f^{\ell}(\mathcal{L})$ is not locally free along the codimension $2$ locus in $\mathcal{C}_{k}$ where $f$ is not smooth. The  jet bundle $J_f^{\ell}(\L):=\bigl(\mathcal{P}_f^{\ell}(\L)\bigr)^{\vee \vee}$, viewed  as a \emph{locally free} replacement of $\mathcal{P}_f^{\ell}(\L)$, sits in a commutative diagram:
$$
 \xymatrix{
         0 \ar[r] & \Omega_f^{\otimes \ell}\otimes \L \ar[r]^{} \ar[d]_{} & \mathcal{P}_f^{\ell}(\L) \ar[d]_{} \ar[r]^{} & \mathcal{P}_f^{\ell-1}(\L) \ar[d] \ar[r] & 0 \\
          0 \ar[r] & \omega_f^{\otimes \ell}\otimes \mathcal{L} \ar[r] &J_f^{\ell}(\mathcal{L}) \ar[r]      & J_f^{\ell-1}(\mathcal{L}) \ar[r] & 0 }
$$
We introduce the vector bundle morphism $\nu_{2}:f^*(\cV)\rightarrow J_f^{2}(\L)$, which for points $[C,A,p]\in \mathcal{C}_{k}$ such that $p\in C_{\mathrm{reg}}$ is just the evaluation morphism $H^0(C,A)\rightarrow H^0(A\otimes \OO_{3p})$. We consider the codimension $2$ locus $Z\subseteq \mathcal{C}_{k}$ where $\nu_2:f^*(\cV)\rightarrow J_f^2(\mathcal{L})$ is not injective. Over the locus of smooth curves, $D_{3}$ is the set-theoretic image of $Z$. A simple local analysis shows that the morphism $\nu_2$ is simply degenerate for each point $[C,A,p]$, where $p\in C_{\mathrm{sing}}$, that is, the divisor $D_0$ also appears (with multiplicity $1$) in the degeneracy locus of $\nu_2$.
Assuming this fact for a moment, via the Porteous formula we obtain:
$$
[D_3]= f_* c_2\left(\frac{J_f^2(\mathcal{L})}{f^*(\cV)}\right)-[D_{0}]\in CH^1(\widetilde{\H}_{k}).
$$

One computes $c_1(J_f^2(\L))=3c_1(\L)+3c_1(\omega_f)$ and $c_2(J_f^2(\L))=3c_1^2(\L)+6c_1(\L)\cdot c_1(\omega_f)+2c_1^2(\omega_f)$,
therefore
$$f_* c_2\left(\frac{J_f^2(\mathcal{L})}{f^*(\cV)}\right)=3\mathfrak{a}+6\mathfrak{b}-3(5k-4)c_1(\cV)+2\kappa_1=6\gamma+2\kappa_1.$$
Recalling that $\kappa_1=12\lambda-[D_0]\in CH^1(\widetilde{\H}_k)$, the claimed formula follows by substitution.

\vskip 3pt

We are left with concluding that $D_0$ appears with multiplicity $1$ in the degeneracy locus $Z$. Let $F:\mathcal{X} \rightarrow B$ be a family of curves of genus $2k-1$ over a $1$-dimensional base $B$, with $\mathcal{X}$ smooth, such that there is a point $b_0\in B$ with $X_b:=F^{-1}(b)$ smooth for $b\in B\setminus \{b_0\}$, whereas $X_{b_0}:=F^{-1}(b_0)$ has a unique node $u\in \mathcal{X}$. Assume we are given $\mathcal{A}\in \mbox{Pic}(\mathcal{X})$  such that $A_b:=\mathcal{A}_{|X_b}\in W^1_k(X_b)$, for each $b\in B$.  We pick a parameter $t\in \mathcal{O}_{B, b_0}$ and $x, y\in \mathcal{O}_{\mathcal{X},u}$, such that $xy=t$ represents the local equation of $\mathcal{X}$ around $u$. Then $\omega_F$ is locally generated by the meromorphic differential $\tau$  given by $\frac{dx}{x}$ outside the divisor $x=0$ and by $-\frac{dy}{y}$ outside the divisor $y=0$. We choose a $\mathbb C[[t]]$-basis $(s_1, s_2)$ of $H^0(\mathcal{X},A)$, where $s_1(u)\neq 0$, whereas $s_2$ vanishes with order $1$ at the node $u$ of $X_{b_0}$, along  both its branches. Passing to germs at $u$, we may assume that  $s_{2,u}=(x+y)s_{1,u}$. Denoting by $\partial:\mathcal{O}_{\mathcal{X}}\rightarrow \omega_F$ the canonical derivation, we note that $\partial(x)=x \tau$ and $\partial(y)=-y\tau$.  Then $Z$  is locally given by the $2\times 2$ minors of the matrix
$$\begin{pmatrix}
1 & 0& 0\\
x+y & x-y & x+y \\
\end{pmatrix},
$$
which proves our claim, that $D_0$ appears with multiplicity $1$.

\end{proof}

\subsection{The divisor $\mathfrak{H}_k^{\mathrm{rk} 4}$}

We fix a cover $[f:C\rightarrow \PP^1]\in \H_k$ and set $A:=f^*(\OO_{\PP^1}(1))$. First we observe that by choosing $[f]$ outside a subset of codimension $2$ in $\H_k$, we may assume that $\omega_C\otimes A^{\vee}$ is very ample. Otherwise by Riemann-Roch there exist points $p,q\in C$ such that $A(p+q)\in W^2_{k+1}(C)$. The Brill-Noether number of this linear series equals $\rho(2k-1,2,k+1)=-1-k<-3$ and it follows from \cite{Ed} that the locus of curves $[C]\in \cM_{2k-1}$ possessing such a linear system has codimension at least $3$ in moduli, which establishes the claim.

\begin{theorem}\label{transvhur}
Fix a general point $[C,A]\in \H_k$. Then the embedded curve $\varphi_{\omega_C\otimes A^{\vee}}:C\hookrightarrow \PP^{k-1}$ lies on no quadrics of rank $4$ or less. It follows that $\mathfrak{H}_k^{\mathrm{rk} 4}$ is a divisor on $\H_k$.
\end{theorem}
\begin{proof}
We choose a polarized $K3$ surface $X$ such that $\mbox{Pic}(X)\cong \mathbb Z\cdot L\oplus \mathbb Z \cdot E$, where $L^2=4k-4$, the curve $E$ is elliptic with $E^2=0$ and $L\cdot E=k$. First we observe that $X$ contains no $(-2)$-curves, hence an effective line bundle $\alpha\in \mbox{Pic}(X)$ must necessarily be nef and satisfy $\alpha^2\geq 0$.

Since $(L-2E)^2=-4$, we compute $\chi(X, L(-2E))=0$. Furthermore, as we have just pointed out $H^0(X,L(-2E))=0$, whereas obviously $H^2(X,L(-2E))=0$, which implies that $H^1(X,L(-2E))=0$, as well. We choose a general curve $C\in |L|$ and set $A:=\OO_C(E)\in W^1_k(C)$. By taking cohomology in the exact sequence $$0\longrightarrow L(-2E)\longrightarrow L^{\otimes 2}(-2E)\longrightarrow \omega_C^{\otimes 2}(-2A)\longrightarrow 0,$$
we obtain an isomorphism $H^0\bigl(X, L^{\otimes 2}(-2E)\bigr)\cong H^0\bigl(C,\omega_C^{\otimes 2}(-2A)\bigr)$. Since clearly, the isomorphism $H^0(X,L(-E))\cong H^0(C,\omega_C(-A))$ also holds, we obtain
$$I_{X,L}(2)\cong I_{C,\omega_C(-A)}(2),$$
so it suffices to show that the embedded surface $\varphi_{L(-E)}:X\hookrightarrow \PP^{k-1}$ lies on no quadric of rank $4$. This amounts to showing that one cannot have a decomposition $L-E=D_1+D_2$, where $D_1$ and $D_2$ are divisor classes on $X$ with $h^0(X,D_i)\geq 2$, for $i=1,2$. Assume we have such a decomposition and write $D_i=x_iC+y_iE$, where $x_1+x_2=1$ and $y_1+y_2=-1$. Since $E$ is nef, we obtain that both $x_1$ and $x_2$ have to be non-negative and we assume $x_1=0$ and $x_2=1$. Then $D_1\equiv y_1E$, therefore $y_1\geq 1$, yielding $y_2\leq -2$, which implies that $h^0(X,D_2)\leq h^0(X,L(-2E))=0$, which leads to a contradiction.
\end{proof}

The divisor $\mathfrak{H}_k^{\mathrm{rk} 4}$ decomposes into  components, depending on the degrees of the pencils $A_1$ and $A_2$ for which the decomposition (\ref{3pen}) holds. For instance, when $\mbox{deg}(A_1)=\mbox{deg}(A)=k$, we obtain the component denoted in \cite{FK}  by $\mathfrak{BN}$ and which consists of pairs $[C,A]\in \H_k$, such that $C$ carries a second pencil of degree $k$. It is shown in \cite{FK} that $\mathfrak{BN}$ has a syzygy-theoretic incarnation that makes reference only to the canonical bundle, being equal to the \mbox{Eagon-Northcott divisor} on $\H_k$ of curves for which $b_{k-1,1}(C,\omega_C)\geq k$. It is an interesting question whether the remaining components of $\mathfrak{H}_k^{\mathrm{rk } 4}$ have a similar intrinsic realization.
\vskip 4pt

We now compute the class of the closure of $\mathfrak{H}_k^{\mathrm{rk} 4}$ inside $\widetilde{\H}_k$:

\begin{theorem}\label{partclass}
The following formula holds: $[\overline{\mathfrak H}_k^{\mathrm{rk} 4}]=A_k^{k-4}\Bigl(\frac{5k+12}{k}\lambda+\frac{k-6}{k}\gamma-[D_0]\Bigr).$
\end{theorem}
\begin{proof}
We are in a  position to apply Theorem \ref{classdiv1} and then $[\overline{\mathfrak H}_k^{\mathrm{rk} 4}]=A_k^{k-4}\Bigl(c_1(\cF)-\frac{4}{k}(2k-3) c_1(\cE)\Bigr)$ and we substitute these Chern classes with the formulas provided in Proposition \ref{chernhurw1}.
\end{proof}

The proof of Theorem \ref{hurwdiv} from the Introduction now follows. We substitute the formula for the class $\gamma$ obtained from Theorem \ref{partclass} in the expression provided by Proposition \ref{trans}, then compare it to the formula for $K_{\widetilde{\H}_k}$.

\vskip 3pt

\noindent \emph{Proof of Theorem \ref{koddim}.}  It is enough to observe that for $k\geq 12$,  the class $7\lambda-\delta_0$ is big on $\mm_{2k-1}$ and there exists an effective divisor of this slope that does not contain $\mbox{Im}(\sigma)=\mm_{2k-1,k}^1$ as a component. This follows from results in \cite{F3} Corollary 0.6, where it is proved that the divisor $\overline{D}_{2k-1,k+1}$ on $\mm_{2k-1}$ already considered in (\ref{nextto}), has support distinct from that of $\mm_{2k-1,k}^1$ and slope $\frac{6k^2+14k+3}{k(k+1)}<7$.
\hfill $\Box$

\end{document}